\documentclass[20pt]{article}

\usepackage{geometry}
\usepackage{amsmath,amsfonts,amssymb,amsthm,amsxtra,bm}
\usepackage{mathrsfs}
\usepackage{cite}
\usepackage{extarrows}
\usepackage{bookmark}
\usepackage{array}
\usepackage{verbatim}

\usepackage[normalem]{ulem}
\usepackage{algorithm}
\usepackage{algorithmic}
\usepackage{setspace}

\usepackage{caption}
\usepackage{graphicx}
\usepackage{caption}
\usepackage{subfig}
\usepackage{color}
\usepackage[dvipsnames]{xcolor}
\colorlet{Mycolor1}{green!20!orange!80!}

\usepackage{multirow}
\usepackage{booktabs}
\usepackage{lscape}
\usepackage{rotating}

\usepackage{epstopdf}
\usepackage{tikz}
\usetikzlibrary{graphs}
\usetikzlibrary{shadows,arrows,positioning,shapes.geometric}
\pgfdeclarelayer{background}
\pgfdeclarelayer{foreground}
\pgfsetlayers{background,main,foreground}

\tikzstyle{R} = [rectangle, rounded corners, minimum width=4cm, minimum height=1cm, text centered, text width=4cm, draw=black]
\tikzstyle{B} = [rectangle, rounded corners, minimum width=11cm, minimum height=1cm, text centered, text width=11cm, draw=black]

\tikzstyle{arrow} = [thick, ->, >=stealth]
\tikzstyle{line} = [draw, thick, color=black!65, -latex']

\newcommand{\rr}[2]{node (p#1) [R] {#2}}
\newcommand{\bb}[2]{node (p#1) [B] {#2}}

\date{}
\newtheorem{definition}{Definition}

\newtheorem{theorem}{Theorem}
\newtheorem{lemma}{Lemma}
\newtheorem{proposition}{Proposition}
\newtheorem{corollary}{Corollary}

\newtheorem{assumption}{Assumption}

\DeclareMathOperator*{\argmax}{argmax}
\DeclareMathOperator*{\argmin}{argmin}
\DeclareMathOperator*{\argsup}{argsup}
\DeclareMathOperator*{\arginf}{arginf}
\usepackage{authblk}
\allowdisplaybreaks[4]

\newtheoremstyle{noparens}
{}{}
{\itshape}{}
{\bfseries}{\bf{.}}
{ }
{\thmname{#1}\thmnumber{ #2}\mdseries\thmnote{ #3}}

\theoremstyle{noparens}

\title{Distributionally robust second-order stochastic dominance constrained optimization with Wasserstein distance}
\author{Yu Mei, Jia Liu, Zhiping Chen\footnote{corresponding author, email: {zchen@mail.xjtu.edu.cn} }}
\affil{School of Mathematics and Statistics, Xi'an Jiaotong University, Xi'an, 710049, P. R. China, Center for Optimization Technique and Quantitative Finance, Xi'an International Academy for Mathematics and Mathematical Technology, Xi'an, 710049, P. R. China}
\begin{document}

\maketitle

\noindent {\bf Abstract}\\
  We consider a distributionally robust second-order stochastic dominance constrained optimization problem. 
  {We require the dominance constraints hold with respect to all probability distributions in a Wasserstein ball centered at the empirical distribution.} We adopt the sample approximation approach to develop a linear programming formulation that provides a lower bound. We propose a novel split-and-dual decomposition framework which provides an upper bound. We {establish quantitative convergency for both lower and upper approximations given some constraint qualification conditions}. To efficiently solve the non-convex upper bound problem, we use a sequential convex approximation algorithm. Numerical evidences on a portfolio selection problem valid the {convergency} and {effectiveness } of the proposed two approximation methods.



\section{Introduction}
Stochastic dominance (SD), originated from economics, is popular in comparing random outcomes. In their pioneering work \cite{D1}, Dentcheva and Ruszczy{\'{n}}ski studied the stochastic optimization problem with univariate SD constraints, where they developed the optimality conditions and duality theory. The commonly adopted univariate SD concepts in stochastic optimization are first-order SD (FSD) and second-order SD (SSD). Researchers have investigated the stochastic optimization problem with FSD constraints from different aspects, such as stability and sensitivity analysis \cite{Dentcheva2009Stability}, mixed-integer linear programming formulations \cite{Luedtke2007New}, and linear programming relaxations \cite{Noyan2006Relaxations}. The stochastic optimization problem with SSD constraints has been intensively studied in quite a few literature. For theoretical foundations, the stability and sensitivity analysis were presented in \cite{Dentcheva2013Stability}. For solution methods, different linear programming formulations were derived in \cite{D1,Luedtke2007New} and the cutting plane methods were adopted in \cite{R1,F1,S1}. The stochastic programs with SD constraints induced by mixed-integer linear recourse were studied in \cite{gollmer2008stochastic} for FSD and in \cite{Gollmer2011A} for SSD. Stochastic optimization problems with multivariate extensions of SD constraints were considered in \cite{Homem2009Cutting,Haskell2017Primal,Noyan2018Optimization}.
There is also a rich literature considering SD under dynamic settings, such as \cite{Escudero2016}. Applications of SD constraints in finance were investigated in \cite{ref2,Hu2017Optimization,Giorgio2019Long}.

A challenge of stochastic programming problems is the accessibility of the true probability distribution of the uncertain parameters. In some practical problems, the true probability distribution sometimes could not be completely observed.  
For this reason, distributionally robust optimization (DRO) models have been proposed to address the lack of {complete information on the true probability distribution}, where the expectations are taken under the worst-case probability distribution in a specific ambiguity set. There are mainly two types of ambiguity sets in the existing literature. The first type is the moment-based ambiguity sets, which is characterized by some moment inequalities
\cite{Xu,Zymler2013Distributionally}. The second type is the distance-based ambiguity sets, which contain all probability distributions close to a nominal distribution measured by some probability metrics, such as Kullback-Leibler divergence \cite{liu2019distributionally}, $\phi$-divergence \cite{Jiang2016Data},
and Wasserstein distance \cite{liu2019distributionally,ji2020data,xie2019distributionally,chen2018data,mei2020data}. 
Esfahani and Kuhn \cite{MohajerinEsfahani2017} estimated a priori probability that the true distribution belongs to the {{Wasserstein ball}} and established finite sample and asymptotic guarantees for the distributionally robust solutions. With the duality theory, the DRO problem with Wasserstein ball can be reformulated as convex programs \cite{MohajerinEsfahani2017,RuiGao2016,zhao2018data}. Such reformulations were then applied to chance-constrained DRO problems \cite{chen2018data,ji2020data,xie2019distributionally}.

Incorporating the basic ideas of SD and distributional robustness, Dentcheva and Ruszczy{\'{n}}ski \cite{bibitem2} first introduced the distributionally robust SD and established the optimality conditions of the stochastic optimization problem with distributionally robust SSD constraints. Since then, a stream of {research} has paid attention to stochastic optimization with distributionally robust SD constraints. 
Dupa{\v{c}}ov{\'a} and Kopa \cite{dupavcova2014robustness} {modeled the ambiguity
of the distribution in FSD by a linear combination of a nominal distribution and a known contamination distribution with the combination parameter being in a parametric uncertainty set.}
Guo, Xu and Zhang \cite{Xu} proposed a discrete approximation scheme for the moment-based ambiguity sets and approximately solved the resulting stochastic optimization problem with distributionally robust SSD constraints. Under a moment-based ambiguity set, Liesi{\"o} et al. \cite{liesio2020portfolio} 
identified optimal portfolios robustly SSD dominating a given benchmark.
Chen and Jiang \cite{chen2018stability} and Zhang et al. \cite{zhang493distributionally} studied stability of DRO problems with $k$th order SD constraints induced by full random recourse. The optimality conditions and duality theory of DRO problems with multivariate SD were discussed in \cite{haskell2016ambiguity,chen2019multivariate}.

As is mentioned above,
SD constrained optimization under distributional ambiguity is {an important class of problems}. 
While the distributionally robust SSD constrained optimization with Wasserstein ball has not been well studied in the existing literature. The {main difficulties of solving such problems} lie in three aspects.
\begin{itemize}
\item The semi-infiniteness {induced from both} the SSD and the distributionally robust 
counterpart are the main challenge. 
\item Distributionally robust SSD constraints are non-smooth 
such that gradient based methods fail to work here.
\item {Compared to moment-based ambiguity sets, the Wasserstein distance contains an extra optimization problem on computing the optimal transportation from the true distribution to the nominal distribution. Such an inner-level optimization problem leads a min-max-min structure and non-convexity of the distributionally robust SSD constraints.}
\end{itemize}
{Therefore, it is quite challenging for us to study the approximation schemes and algorithms for the distributionally robust SSD constrained optimization problem {with Wasserstein ball}. Thanks to the rapid development recently on the strong duality theory of DRO problems with Wasserstein ball \cite{MohajerinEsfahani2017,RuiGao2016}, we have a chance to show in this paper efficient approximation methods for such SSD constrained problem.} 

In detail, we first utilize the strong duality results for DRO problem with Wasserstein ball in \cite{RuiGao2016} to derive a reformulation of distributionally robust SSD constraints. Then we adopt the sampling approach to approximate the infinitely many constraints by finitely many constraints and develop a linear programming formulation which is a lower bound approximation. {We further analyze the quantitative convergency of the lower bound approximation.} To overcome the `curse of dimensionality' of the linear programming approximation, we propose a novel split-and-dual decomposition framework. We separate the support set of the parameter in distributionally robust SSD constraints into finite sub-intervals. For each sub-interval, we exchange the order of the supremum operator and the expectation operator to get an upper bound approximation. We prove that the optimal value of the upper bound approximation converges to that of the original problem as the number of sub-intervals goes to infinity and we quantitatively estimate the approximation error. As the derived upper bound approximation problem is non-convex, we apply the sequential convex approximation method to solve it.

{This paper improves results in quite a few papers. Specifically, we extend the DRO with Wasserstein ball \cite{MohajerinEsfahani2017,chen2018data,ji2020data,RuiGao2016,xie2019distributionally} to a more complicated case with infinitely many constraints induced by SSD.} Compared with robust SD constrained optimization problems in \cite{Xu,zhang493distributionally,liesio2020portfolio,chen2019multivariate}, we study Wasserstein ball rather than moment-based ambiguity sets. The main contributions of this paper include:
\begin{itemize}
  \item We derive a lower bound approximation of the distributionally robust SSD constrained optimization with Wasserstein ball by the sample approximation approach, 
  {and establish the quantitative convergency of the approximation problem.}
  \item We propose a novel split-and-dual decomposition framework, which provides an upper bound approximation of the problem. As far as we know, the upper bound approximations of SD constrained problems are seldom studied in existing literature. {We prove the convergency of the approximation approach and quantitatively estimate the approximation error when the number of sub-intervals is sufficiently large.}
\end{itemize}
While preparing this paper for publication, we became aware of an independent work by Peng and Delage
\cite{PeD20} on distributionally robust SSD constraints with Wasserstein ball.
Peng and Delage \cite{PeD20} formulated the distributionally robust SSD constrained problem as a multistage robust optimization problem, and proposed a tractable conservative approximation that exploits finite adaptability and a scenario-based lower bounding problem with nice numerical feasibility.
We distinguish our work from Peng and Delage \cite{PeD20} in a different split-and-dual decomposition framework for the upper bound and detailed quantitative convergency analysis.

The rest of this paper is organized as follows. In section \ref{2}, we introduce the distributionally robust SSD constrained optimization problem. In section \ref{3}, we adopt the sampling approximation approach to obtain the lower bound approximation and establish the quantitative convergency. In section \ref{3.2}, we propose a split-and-dual decomposition framework to derive the upper bound approximation, whose optimal value can be obtained by solving a sequence of second-order cone programming problems. We also quantitatively estimate the approximation error when the number of sub-intervals is sufficiently large. Numerical evidences valid the {convergency}  and {effectiveness}  of the proposed approximation methods in Section \ref{4}. Section \ref{5} concludes the paper.

\section{Preliminaries}\label{2}

\subsection{Distributionally robust second-order stochastic dominance}
First we introduce some notations. Let $\mathcal{U}$ be the set of all non-decreasing and concave {utility} functions $u:\mathbb{R} \to \mathbb{R}$. We use $(\cdot)_+=\max\{\cdot,0\}$
to denote the positive part function. {Let $d(x,A):=\inf_{y\in A}\|x-y\|$ be the distance from a point $x$ to a set $A$. Denote the deviation of a set $A$ from another set $B$ by $\mathbb{D}(A,B)=\sup_{x\in A}d(x,B)$ and the Hausdorff distance between $A$ and $B$ by $\mathbb{H}(A,B)=\max \{\mathbb{D}(A,B),\mathbb{D}(B,A) \}$.} Let $(\Omega,\mathscr{F})$ be a measurable space with $\mathscr{F}$ being the Borel $\sigma$-algebra on $\Omega$, and $\mathscr{P}$ be the set of all probability measures on $(\Omega,\mathscr{F})$.

Before introducing the distributionally robust SSD, we recall the definition of classic SSD. Consider the random variables $X$ and $Y$ on a probability space $(\Omega, \mathscr{F},P)$ {with finite first order moments}, here $P \in \mathscr{P}$ is the true distribution. We say that $X$ stochastically dominates $Y$ in the second order, denoted by $X \succeq_{(2)}^{P} Y$, if $\mathbb{E}_{P}[u(X)] \ge \mathbb{E}_{P}[u(Y)],~\forall u \in \mathcal{U}$. $X \succeq_{(2)}^{P} Y$ is equivalent to
\begin{equation}\label{equ:1}
\mathbb{E}_{P}[(\eta-X)_+-(\eta-Y)_+]\le0,~\forall \eta \in \mathbb{R}.
\end{equation}
Let $\mathcal{Y}$ be the set of all realizations of the random variable $Y$. It has been shown in \cite[Proposition 1]{liu2016approximation} that \eqref{equ:1} is equivalent to
\begin{equation}\label{equ:2}
\mathbb{E}_{P}[(\eta-X)_+-(\eta-Y)_+]\le0,~\forall \eta \in \mathcal{Y}.
\end{equation}

{In some data-driven problems, it is difficult to obtain the complete information about the true probability measure $P$}. To address this issue, 
Dentcheva and Ruszczy{\'{n}}ski \cite{bibitem2} introduced distributionally robust SSD by considering an ambiguity set of probability measures instead of $P$.
\begin{definition}
  $X$ dominates $Y$ robustly in the second order over a set of probability measures $\mathcal{Q} \subset \mathscr{P}$,
  denoted by $X \succeq^{\mathcal{Q}}_{(2)} Y$, if
  \begin{equation*}
    \mathbb{E}_{P}[u(X)] \ge \mathbb{E}_{P}[u(Y)],~\forall u \in \mathcal{U},~\forall P \in \mathcal{Q}.
  \end{equation*}
\end{definition}

In the rest of this paper, we investigate the following distributionally robust SSD constrained optimization problem
\begin{flalign*}
(P_{SSD})&
&\min \limits_{z \in Z}  \quad &  f(z)& \\
& &\quad\text{s.t.} \,\quad & z^T \xi \succeq_{(2)}^{\mathcal{Q}} z_0^T \xi, &
\end{flalign*}
where $f$ is proper and continuous, $\xi$ denotes the random vector, $Z\subset \mathbb{R}^n$ is a bounded polyhedral set, and $z_0 \in Z$ is a given benchmark.
From \eqref{equ:1}, problem ($P_{SSD}$) can be rewritten as
\begin{equation}\label{shenme}\begin{array}{cl}
\min \limits_{z \in Z}  \quad & f(z)  \\
\quad\text{s.t.} \ \ \quad & \mathbb{E}_{P}[(\eta- z^T \xi)_+-(\eta-z_0^T \xi)_+]\le0,~\forall \eta \in \mathbb{R},~\forall P \in \mathcal{Q}.
\end{array}\end{equation}
{We can observe that the semi-infiniteness of constraints in problem \eqref{shenme} arises from} $\eta \in \mathbb{R}$ and $P \in \mathcal{Q}$, induced from the SSD constraints and the distributionally robust ambiguity set, respectively. Moreover, the constraint functions in problem \eqref{shenme} are non-smooth as $(\cdot)_+$ is involved. Therefore, problem \eqref{shenme}, as well as problem ($P_{SSD}$), is hard to solve. To reduce the difficulties in solving problem ($P_{SSD}$), we firstly assume that the support set $\Xi$ has a polyhedral structure. The polyhedral structure of $\Xi$, also assumed in \cite[Corollary 5.1]{MohajerinEsfahani2017}, contributes to applying the duality theory of second-order conic programming when deriving the upper bound approximation later in this paper.

\begin{assumption}\label{ass:Xi}
{
$\Xi$ is polyhedral, i.e.,
$\Xi=\{\xi\in\mathbb{R}^n \mid C \xi \le d\},$
where $C \in \mathbb{R}^{l \times n}$, $d \in \mathbb{R}^l$, and $z_{0}^{T}\Xi:=\{z_0^T\xi \mid \xi \in \Xi\}$ is a compact set.}
\end{assumption}

{If $\Xi$ is a polyhedron, then Assumption \ref{ass:Xi} holds automatically.}
Keeping in mind the equivalence of \eqref{equ:1} and \eqref{equ:2}, problem \eqref{shenme} can be formulated as
\begin{equation}\label{shenme2}
    \begin{array}{cl}
\min \limits_{z \in Z}  \quad & f(z)  \\
\quad\text{s.t.} \ \ \quad & \mathbb{E}_{P}[(\eta- z^T \xi)_+-(\eta-z_0^T \xi)_+]\le0,~\forall \eta \in \mathcal{R}:=z_{0}^{T}\Xi,~\forall P \in \mathcal{Q}.
\end{array}
\end{equation}
We denote the smallest and largest numbers in $\mathcal{R}$ by $\mathcal{R}_{{\min}}$ and $\mathcal{R}_{{\max}}$, respectively, thus, 
$\mathcal{R}=[\mathcal{R}_{{\min}},\mathcal{R}_{{\max}}]$.

\subsection{Data-driven Wasserstein ambiguity set}

In this section, we introduce the data-driven Wasserstein ambiguity set $\mathcal{Q}$ and recall a fundamental duality result in DRO problems with Wasserstein ball \cite{MohajerinEsfahani2017,RuiGao2016,zhao2018data}.

Let $\mathscr{P}(\Xi)$ be the space of all probability measures $P$ supported on $\Xi$ with $\mathbb{E}_{P}[\|\xi\|]<\infty$. We consider 
$1$-Wasserstein distance, also known as Kantorovich metric.

\begin{definition}\label{def:0928}
The \emph{Kantorovich metric} $d_K\colon \mathscr{P}(\Xi) \times \mathscr{P}(\Xi) \to \mathbb{R}_+$ is defined via
\begin{equation*}
d_K(P,Q):=\inf_{\pi} \Bigg\{ \int_{\Xi^2}\|\xi_1-\xi_2\|\pi (d\xi_1, d\xi_2):
\begin{array}{ll}
 \pi \text{ is a joint distribution of }\xi_1\text{ and }\xi_2 \\
 \text{with marginals }P\text{ and }Q,\text{ respectively}
 \end{array}
 \Bigg\}.
\end{equation*}
\end{definition}

{Kantorovich metric can be written in a pseudo metric form \cite{ChenSunXu2021}.}

\begin{proposition}\label{def:Kantorovich}
{Let $\mathscr{G}$ be the set of all Lipschitz continuous functions $h:\Xi \to \mathbb{R}$ with modulus $1$. Then 
\begin{equation*}
d_K(P,Q):=\sup_{h\in\mathscr{G}} \left| \int_{\Xi}h(\xi)P(d\xi)-\int_{\Xi}h(\xi)Q(d\xi) \right|.
\end{equation*}}
\end{proposition}

{Given $\mathcal{P},\mathcal{Q}\subset \mathscr{P}(\Xi)$, define the deviation of $\mathcal{P}$ from $\mathcal{Q}$ by $\mathbb{D}_K(\mathcal{P},\mathcal{Q}):=\sup_{P\in \mathcal{P}}$ $\inf_{Q\in \mathcal{Q} } d_K(P,Q)$, and the Hausdorff distance between $\mathcal{P}$ and $\mathcal{Q}$ by $\mathbb{H}_K(\mathcal{P},\mathcal{Q}):=\max \{\mathbb{D}_K(\mathcal{P},\mathcal{Q}),$ $\mathbb{D}_K(\mathcal{Q},\mathcal{P})\}$.}

Given $N$ observations $\{\widehat{\xi}_i\}_{i=1}^N$ of $\xi$, we define the data-driven Wasserstein ambiguity set $\mathcal{Q}$ as a ball centered at the empirical distribution $\widehat{P}_N=\frac{1}{N}\sum_{i=1}^{N}\delta_{\widehat{\xi}_i}$,
\begin{equation}\label{equ:Q}
\mathcal{Q}:=\{P \in \mathscr{P}(\Xi):d_K(P,\widehat{P}_N)\le \epsilon\},
\end{equation}
where $\epsilon$ is a prespecified robust radius. Esfahani and Kuhn \cite{MohajerinEsfahani2017} proved that with any prescribed $\beta \in (0,1)$, by appropriately defining $\epsilon(\beta)$, the true distribution $P$ belongs to $\mathcal{Q}$ with the confidence level $1-\beta$.

Under some mild conditions, strong duality results of DRO problems {with Wasserstein ball} have been established in \cite[Theorem 4.2]{MohajerinEsfahani2017},
\cite[Corollary 2]{RuiGao2016} and \cite[Proposition 2]{zhao2018data}.

\begin{lemma}\label{lemma}
{If $\Psi(\xi)$ is proper, continuous,
and for some $\zeta\in\Xi$, the growth rate $\kappa:=\limsup_{\|\xi-\zeta\|\to \infty}\frac{\Psi(\xi)-\Psi(\zeta)}{\|\xi-\zeta\|}<\infty$},
then the optimal values of
\begin{equation*}
\sup \limits_{P \in \mathscr{P}(\Xi)}\left\{\int_{\Xi}\Psi(\xi)P(d\xi):d_K(P,\widehat{P}_N)\le \epsilon \right\}
\end{equation*}
and
\begin{equation}\label{1001}
\min_{\lambda\ge 0}\left\{\lambda \epsilon+ \frac{1}{N}\sum_{i=1}^{N}\sup_{\xi \in \Xi}[\Psi(\xi)-\lambda \|\xi-\widehat{\xi}_i\|] \right\}
\end{equation}
are equal. {Moreover, the optimal solution set of \eqref{1001} is nonempty and compact.}
\end{lemma}

\begin{proof}
{The strong duality was established in \cite[Theorem 1]{RuiGao2016}. We only need to prove the nonemptiness and compactness of the optimal solution set of \eqref{1001}.
  Let
  $
  f_D(\lambda)\!:=\!\lambda \epsilon +\! \frac{1}{N}\sum_{i=1}^N\! \sup_{\xi \in \Xi}[\Psi(\xi)-\!\lambda \|\xi-\widehat{\xi}_i\| ].
  $
 It is easy to see that $f_D$ is lower semi-continuous and convex, and $f_D(\lambda)\ge\lambda \epsilon + \frac{1}{N}\sum_{i=1}^N \Psi(\widehat{\xi}_i),\forall \lambda\ge 0$.
  Since $\Psi$ is proper, 
  $f_D$
  is bounded from below on bounded sets and
  $$
  \liminf_{\lambda \to \infty} \frac{f_D(\lambda)}{\lambda}\ge  \liminf_{\lambda \to \infty} \frac{\lambda \epsilon + \frac{1}{N}\sum_{i=1}^N \Psi(\widehat{\xi}_i)}{\lambda}=\epsilon>0.
  $$
  This means that $f_D$ is level-coercive \cite[Definition 3.25]{Rockafellar2009Variational}.
  }

  {
  Next, we prove that $f_D$ is proper by showing $f_D(\kappa)<\infty$ at $\kappa:=\!\limsup\limits_{\|\xi-\zeta\|\to \infty}$ $({\Psi(\xi)-\Psi(\zeta)})/{\|\xi-\zeta\|}$. Consider for each $i=1,\cdots,N$, the optimal value $\vartheta_i^*$ of problem $\sup_{\xi \in \Xi}f_i(\xi)$ $:=\Psi(\xi)-\kappa \|\xi-\widehat{\xi}_i\|$. Assume there exists a sequence $\{\xi_i^{\iota}\}_{\iota=1}^{\infty}\subset \Xi    $ such that $f_i(\xi_i^{\iota}) \to \vartheta_i^*$. We then have two cases.
  Case 1: $\|\xi_i^{\iota}\|\to\infty$. As the growth rate $\kappa$ does not depend on the choice of $\zeta$ \cite[lemma 4]{RuiGao2016}, we can select $\xi=\xi_i^{\iota}$ and $\zeta=\widehat{\xi_i}$, and thus have
  $\vartheta_i^*=\lim_{{\iota}\to \infty}\Psi(\xi_i^{\iota})-\kappa \|\xi_i^{\iota}-\widehat{\xi}_i\| \le \Psi(\widehat{\xi}_i)<\infty $.
  Case 2: $\xi_i^{\iota} \to \xi_i^*$ with $\|\xi_i^*\|<\infty$. 
  Since $\Xi$ is closed, $\xi_i^*\in\Xi$. By continuity of $\Psi$, we have $\vartheta_i^*=\lim_{{\iota}\to \infty}\Psi(\xi_i^{\iota})-\kappa \|\xi_i^{\iota}-\widehat{\xi}_i\|=\Psi(\xi_i^*)-\kappa \|\xi_i^*-\widehat{\xi}_i\|<\infty$.
  Therefore, $f_D(\kappa)=\kappa \epsilon +\! \frac{1}{N}\sum_{i=1}^N \vartheta_i^*<\infty$ and thus $f_D$ is proper. By \cite[Corollary 3.27]{Rockafellar2009Variational}, $f_{D}$ is level-bounded. Therefore, it is known from \cite[Theorem 1.9]{Rockafellar2009Variational} that $\argmin_{\lambda} f_D(\lambda)$ is nonempty and compact.}
\end{proof}

\subsection{Flowchart of the lower and upper bounds approximation schemes}

Later on, we will derive for problem $(P_{SSD})$ a lower bound approximation in Section \ref{3} and an upper bound approximation in Section \ref{3.2}. The relationship of formulations in intermediate steps of the two approximation schemes is illustrated in Figure \ref{fig:111}.

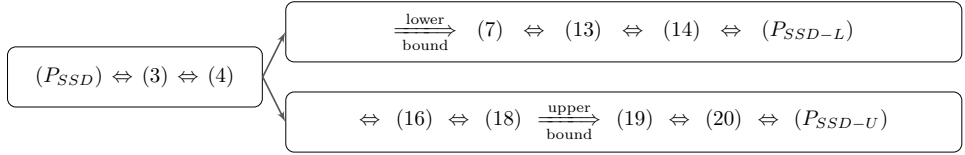
\begin{figure}[h]
\setstretch{1}
\scalebox{0.8}{
\begin{tikzpicture}
\path \rr{1}{
$(P_{SSD})\Leftrightarrow \eqref{shenme}\Leftrightarrow \eqref{shenme2}$ };
\path (p1.east) + (6,0.75) \bb{2}{$ \xLongrightarrow[\text{bound}]{\text{lower}} \eqref{0925} \Leftrightarrow \eqref{flnew} \Leftrightarrow \eqref{problem3-8} \Leftrightarrow (P_{SSD-L})$};
\path (p1.east) + (6,-0.75) \bb{3}{$\Leftrightarrow \eqref{shenme3}  \Leftrightarrow \eqref{problem4-3} \xLongrightarrow[\text{bound}]{\text{upper}} \eqref{222} \Leftrightarrow \eqref{problem4-6} \Leftrightarrow (P_{SSD-U})$};
\path [line] (p1.east) -- node [above] {} (p2.west);
\path [line] (p1.east) -- node [above] {} (p3.west);
\end{tikzpicture}}
\caption{The flowchart 
of the two approximation schemes.}
\label{fig:111}
\end{figure}

The key reformulation or approximation steps in the two approximation schemes can be summarized as follows:

1) Reformulations ($P_{SSD}$) $\Leftrightarrow$ \eqref{shenme} and \eqref{shenme} $\Leftrightarrow$ \eqref{shenme2} are due to the definition of distributionally robust SSD.

2) Approximation \eqref{shenme2} $\xLongrightarrow[\text{bound}]{\text{lower}}$ \eqref{0925} comes from the finite sample approximation; Reformulation \eqref{0925} $\Leftrightarrow$ \eqref{flnew} is due to the duality theory of DRO problems with  Wasserstein ball from Lemma \ref{lemma}; Reformulations \eqref{flnew} $\Leftrightarrow$ \eqref{problem3-8} and \eqref{problem3-8} $\Leftrightarrow$ ($P_{SSD-L}$) are obtained by adding auxiliary variables.

3) We propose a split-and-dual decomposition framework for the upper bound approximation. In detail,
\eqref{shenme2} $\Leftrightarrow$ \eqref{shenme3} is a rewrite; We split the interval $\mathcal{R}$ into sub-intervals in the reformulation \eqref{shenme3} $\Leftrightarrow$ \eqref{problem4-3}; We exchange the order of the expectation and supremum to derive the upper bound approximation \eqref{problem4-3} $\xLongrightarrow[\text{bound}]{\text{upper}}$ \eqref{222}; Reformulations \eqref{222}$\Leftrightarrow$\eqref{problem4-6} is due to the duality theory of DRO problems with Wasserstein ball from Lemma \ref{lemma}; Reformulation \eqref{problem4-6} $\Leftrightarrow$ ($P_{SSD-U}$) is due to the strong duality of second-order cone programming.

\section{Lower bound approximation of distributionally robust SSD constrained optimization}\label{3}

{In order to tackle the semi-infiniteness arising from the constraints in problem \eqref{shenme2}, we consider the approximation of the sets $\mathcal{Q}$ and $\mathcal{R}$. Let $\Xi_{\mathcal{N}} =\{\bar{\xi_j}\}_{j=1}^{\mathcal{N}}$ be a set of finite samples in $\Xi$ and $\varGamma_{\mathcal{M}}=\{\eta_k\}_{k=1}^{\mathcal{M}}$ be a set of finite samples in $\mathcal{R}=[\mathcal{R}_{{\min}},\mathcal{R}_{{\max}}]$, here $\mathcal{N}$ and $\mathcal{M}$ denote the sample sizes.
We then approximate the ambiguity set $\mathcal{Q}$ by the following Wasserstein ball:
\begin{equation*}
    \mathcal{Q}_{\mathcal{N}}:=\{P \in \mathscr{P}(\Xi_{\mathcal{N}}):d_K(P,\widehat{P}_N)\le \epsilon\}.
\end{equation*}
Different from $\mathcal{Q}$ which covers both discrete and continuous probability measures supported on $\Xi$, $\mathcal{Q}_{\mathcal{N}}$ only contains discrete probability measures supported on $\Xi_{\mathcal{N}}$. It is easy to see $\mathcal{Q}_{\mathcal{N}} \subset \mathcal{Q}$ and $\varGamma_{\mathcal{M}} \subset \mathcal{R}$. Therefore, we have a lower bound approximation of problem \eqref{shenme2}:
\begin{equation}\label{0925}\begin{array}{cl}
\min \limits_{z \in Z}  \quad & f(z)  \\
\quad\text{s.t.} \ \ \quad & \mathbb{E}_{P}[(\eta- z^T \xi)_+-(\eta-z_0^T \xi)_+]\le0,~\forall \eta \in \varGamma_{\mathcal{M}},~\forall P \in \mathcal{Q}_{\mathcal{N}}.
\end{array}\end{equation}
}%

{
In subsection \ref{subsec:3.1}, we establish the quantitative convergency for problem \eqref{0925} in terms of the feasible set, the optimal value, and the optimal solution set. Then in subsection \ref{subsec:3.2}, we show how problem \eqref{0925} can be reformulated as a linear programming problem, and the computational efficiency for large sample sizes can be further improved by the cutting-plane method.}

\subsection{\texorpdfstring{Quantitative analysis of the lower approximation}{ }}\label{subsec:3.1}
{We denote the feasible sets of problem \eqref{shenme2} and problem \eqref{0925} by $\mathcal{F}$ and $\mathcal{F}_{\mathcal{N},\mathcal{M}}$,
the optimal solution sets by $\mathcal{S}$ and $\mathcal{S}_{\mathcal{N},\mathcal{M}}$, and the optimal values by $v$ and $v_{\mathcal{N},\mathcal{M}}$, respectively. To establish the quantitative convergency for problem \eqref{0925}, we need the following Slater constraint qualification.}
{
\begin{assumption}\label{ass:0926-Slater}
There exist a point $\bar{z}\in Z$ and a constant $\theta>0$ such that
\begin{equation*}
    \sup_{\eta \in \mathcal{R}} \sup_{P \in \mathcal{Q}}\mathbb{E}_{P}[(\eta-\bar{z}^T\xi)_{+}-(\eta-z_0^T \xi)_{+}] < -\theta.
\end{equation*}
\end{assumption}}
{Let $L:=\sup_{z \in Z}(\|z\|+\|z_0\|)<\infty$, since $Z$ is compact. Denote the Hausdorff distance between $\Xi$ and $
\Xi_{\mathcal{N}}$ by $\alpha_{\mathcal{N}}:=\sup_{\xi \in \Xi} \inf_{\xi'\in \Xi_{\mathcal{N}}}\|\xi- \xi' \|$, and the Hausdorff distance between $\mathcal{R}$ and $\varGamma_{\mathcal{M}}$ by $\gamma_{\mathcal{M}}:=\sup_{\eta \in \mathcal{R}} \inf_{\eta'\in\varGamma_{\mathcal{M}}}|\eta-\eta'|$.}
{
\begin{assumption}\label{ass:0926-alpha-gamma}
$\lim_{\mathcal{N}\to \infty}\alpha_{\mathcal{N}}=0 $ and $\lim_{\mathcal{M}\to \infty}\gamma_{\mathcal{M}}=0$.
\end{assumption}}
\begin{theorem}\label{the:0927-1}
{
Given Assumptions \ref{ass:0926-Slater} and \ref{ass:0926-alpha-gamma}, the following assertions hold.}

{
\emph{(i)} For any $\mathcal{N}$ and $\mathcal{M}$,
\begin{equation*}
    \mathbb{H}(\mathcal{F}_{\mathcal{N},\mathcal{M}},\mathcal{F})\le \frac{2D_Z}{\theta}(L\alpha_{\mathcal{N}}+\gamma_{\mathcal{M}}
),
\end{equation*}
where $D_Z$ denotes the diameter of $Z$ and $\theta$ is defined as in Assumption \ref{ass:0926-Slater}.}

{
\emph{(ii)}
  $\lim_{\mathcal{N} \to \infty, \atop \mathcal{M} \to \infty} v_{\mathcal{N},\mathcal{M}}=v$ and $\limsup_{\mathcal{N} \to \infty, \atop \mathcal{M} \to \infty} \mathcal{S}_{\mathcal{N},\mathcal{M}} \subset \mathcal{S}$.}

{
\emph{(iii)} If, in addition, the objective function $f$ is Lipschitz continuous with modulus $L_f$, then for any $\mathcal{N}$ and $\mathcal{M}$,
\begin{equation*}
    |v_{\mathcal{N},\mathcal{M}}-v|\le \frac{2D_ZL_f}{\theta}(L\alpha_{\mathcal{N}}+\gamma_{\mathcal{M}}).
\end{equation*}
Moreover, if problem \eqref{shenme2} satisfies the second-order growth condition at the optimal solution set $\mathcal{S}$, i.e., there exists a positive constant $\rho$ such that
$$
    f(z)-v\ge \rho d(z,\mathcal{S})^2,~\forall z\in \mathcal{F},
$$
then for sufficiently large $\mathcal{N}$ and $\mathcal{M}$,
\begin{equation}\label{equ:0927-1}
    \mathbb{D}(\mathcal{S}_{\mathcal{N},\mathcal{M}},\mathcal{S}) \le \Big(\sqrt{\frac{2D_Z}{\theta}}+\sqrt{\frac{4L_fD_Z}{\rho\theta}} \Big) \sqrt{L\alpha_{\mathcal{N}}+\gamma_{\mathcal{M}}}.
\end{equation}
}
\end{theorem}

\begin{proof}
{
  (i) We write $\phi(\eta,z,\xi)=(\eta-z^T\xi)_{+}-(\eta-z_0^T \xi)_{+}$, $w(z)=\sup_{\eta \in \mathcal{R}} \sup_{P \in \mathcal{Q}}$ $\mathbb{E}_{P}[\phi(\eta,z,\xi)]$, and $w_{\mathcal{N},\mathcal{M}}(z)=\sup_{\eta' \in \varGamma_{\mathcal{M}}}\! \sup_{P' \in \mathcal{Q}_{\mathcal{N}}}\!\mathbb{E}_{P'}[\phi(\eta',z,\xi)]$. Since $\mathcal{Q}_{\mathcal{N}} \subset \mathcal{Q}$ and $\varGamma_{\mathcal{M}} \subset \mathcal{R}$, then for any $z \in Z$,
  \begin{equation}\label{equ:0926-1}
   w_{\mathcal{N},\mathcal{M}}(z) \le w(z).
  \end{equation}
  One can observe that $\phi$ is uniformly Lipschitz continuous w.r.t. $\xi$ with modulus $L$ and also uniformly Lipschitz continuous w.r.t. $\eta$ with modulus $2$. Then by Proposition \ref{def:Kantorovich}, we have
\begin{flalign*}
  &|\mathbb{E}_{P}[\phi(\eta,z,\xi)]-\mathbb{E}_{P'}[\phi(\eta',z,\xi)]|\\
\le~ &|\mathbb{E}_{P}[\phi(\eta,z,\xi)]-\mathbb{E}_{P'}[\phi(\eta,z,\xi)]|+ |\mathbb{E}_{P'}[\phi(\eta,z,\xi)]-\mathbb{E}_{P'}[\phi(\eta',z,\xi)]|\\
\le~ & Ld_K(P,P')+ 2|\eta-\eta'|.
\end{flalign*}
Therefore, for any $z \in Z$,
\begin{flalign}
 & w(z)-w_{\mathcal{N},\mathcal{M}}(z)
 \le  \sup_{\eta \in \mathcal{R}} \inf_{\eta'\in \varGamma_{\mathcal{M}}}
\sup_{P \in \mathcal{Q}} \inf_{P' \in \mathcal{Q}_{\mathcal{N}}}
\Big(
\mathbb{E}_{P}[\phi(\eta,z,\xi)]-\mathbb{E}_{P'}[\phi(\eta',z,\xi)]
\Big) \nonumber \\
\le & \sup_{\eta \in \mathcal{R}} \inf_{\eta'\in \varGamma_{\mathcal{M}}}
\sup_{P \in \mathcal{Q}} \inf_{P' \in \mathcal{Q}_{\mathcal{N}}}
\Big(
Ld_K(P,P')+ 2|\eta-\eta'|
\Big) \label{equ:0926-2}\\
= & \sup_{P \in \mathcal{Q}} \inf_{P' \in \mathcal{Q}_{\mathcal{N}}} Ld_K(P,P')+\sup_{\eta \in \mathcal{R}} \inf_{\eta'\in \varGamma_{\mathcal{M}}}2|\eta-\eta'|
=   L\mathbb{D}_K(\mathcal{Q}, \mathcal{Q}_{\mathcal{N}})+2\gamma_{\mathcal{M}}. \nonumber
\end{flalign}
By \cite[Theorem 2]{ChenSunXu2021}, we have $\mathbb{H}_K(\mathcal{Q}, \mathcal{Q}_{\mathcal{N}})\le 2\alpha_{\mathcal{N}}$.
Thus, combining inequalities \eqref{equ:0926-1} and \eqref{equ:0926-2} gives
\begin{equation*} 
  \sup_{z \in Z}|w(z)-w_{\mathcal{N},\mathcal{M}}(z)|\le L \mathbb{H}_K(\mathcal{Q}, \mathcal{Q}_{\mathcal{N}})
  +2\gamma_{\mathcal{M}} \le
  2L \alpha_{\mathcal{N}}
  +2\gamma_{\mathcal{M}}.
\end{equation*}
}

{With Assumption \ref{ass:0926-Slater}, we can apply Robinson's error bound for convex inequality system \cite{Robinson1975} and obtain $d(z,\mathcal{F})\le \frac{D_Z}{\theta}[w(z)]_+,~\forall z\in Z$. On one hand, for any $z\in \mathcal{F}_{\mathcal{N},\mathcal{M}}$, we have
\begin{flalign*}
d(z,\mathcal{F})&\le \frac{D_Z}{\theta}[w(z)]_+
\le \frac{D_Z}{\theta}(|w(z)-w_{\mathcal{N},\mathcal{M}}(z)|+[w_{\mathcal{N},\mathcal{M}}(z)]_+) \\
& = \frac{D_Z}{\theta}|w(z)-w_{\mathcal{N},\mathcal{M}}(z)|
\le \frac{2D_Z}{\theta}(L \alpha_{\mathcal{N}} +\gamma_{\mathcal{M}}).
\end{flalign*}
This implies
\begin{equation}\label{equ:0926-4}
    \mathbb{D}(\mathcal{F}_{\mathcal{N},\mathcal{M}},\mathcal{F})\le \frac{2D_Z}{\theta}(L \alpha_{\mathcal{N}} +\gamma_{\mathcal{M}}).
\end{equation}
Since $w_{\mathcal{N},\mathcal{M}}(z) \le w(z),\forall z\in Z$, Assumption \ref{ass:0926-Slater} implies $w_{\mathcal{N},\mathcal{M}}(\bar{z})<-\theta$. This means that the convex inequality constraint $w_{\mathcal{N},\mathcal{M}}(z)\le 0$ also satisfies the Slater constraint qualification.
On the other hand, for any $z\in \mathcal{F}$, we have
\begin{equation*}
    d(z,\mathcal{F}_{\mathcal{N},\mathcal{M}}) \le \frac{D_Z}{\theta}[w_{\mathcal{N},\mathcal{M}}(z)]_+ \le
 \frac{D_Z}{\theta}|w_{\mathcal{N},\mathcal{M}}(z)-w(z)|
\le \frac{2D_Z}{\theta}(L \alpha_{\mathcal{N}} +\gamma_{\mathcal{M}}),
\end{equation*}
which implies
\begin{equation}\label{equ:0926-5}
    \mathbb{D}(\mathcal{F},\mathcal{F}_{\mathcal{N},\mathcal{M}})\le \frac{2D_Z}{\theta}(L \alpha_{\mathcal{N}} +\gamma_{\mathcal{M}}).
\end{equation}
Inequalities \eqref{equ:0926-4} and \eqref{equ:0926-5} mean conclusion (i).
}

(ii) Let $\bar{f}(z)=f(z)+\delta_{\mathcal{F}}(z)$ and
  $\bar{f}_{\mathcal{N},\mathcal{M}}(z)=f(z)+\delta_{\mathcal{F}_{\mathcal{N},\mathcal{M}}}(z)$.
  Since $\mathcal{F}_{\mathcal{N},\mathcal{M}}$ converges to $\mathcal{F}$, then by \cite[Proposition 7.4(f)]{Rockafellar2009Variational}, $\delta_{\mathcal{F}_{\mathcal{N},\mathcal{M}}}$ epi-converges to $\delta_{\mathcal{F}}$ as $\mathcal{N} \to \infty,\mathcal{M} \to \infty$. As $f$ is continuous and finite, we obtain by \cite[Exercise 7.8]{Rockafellar2009Variational} that
  $\bar{f}_{\mathcal{N},\mathcal{M}}=f+\delta_{\mathcal{F}_{\mathcal{N},\mathcal{M}}}$ epi-converges to $\bar{f}=f+\delta_{\mathcal{F}}$
  when $\mathcal{N} \to \infty,\mathcal{M} \to \infty$. As $\mathcal{F},~\mathcal{F}_{\mathcal{N},\mathcal{M}}$ are closed and $f$ is continuous, $\bar{f}_{\mathcal{N},\mathcal{M}}$ and $\bar{f}$ are lower semi-continuous. Moreover, since $\bar{f}_{\mathcal{N},\mathcal{M}}$ and $\bar{f}$ are proper, it can then be deduced from \cite[Theorem 7.33]{Rockafellar2009Variational} that
  $v=\lim_{\mathcal{N} \to \infty\atop\mathcal{M} \to \infty} v_{\mathcal{N},\mathcal{M}}$ and
  $\limsup_{\mathcal{N} \to \infty \atop \mathcal{M} \to \infty} \mathcal{S}_{\mathcal{N},\mathcal{M}}\subset\mathcal{S}$.

{
(iii) Let $z\in \mathcal{S}$ and $z_{\mathcal{N},\mathcal{M}} \in \mathcal{S}_{\mathcal{N},\mathcal{M}}$. By definition of $\mathbb{D}(\mathcal{F}_{\mathcal{N},\mathcal{M}},\mathcal{F})$, $d(z_{\mathcal{N},\mathcal{M}},\mathcal{F})\le \mathbb{D}(\mathcal{F}_{\mathcal{N},\mathcal{M}},\mathcal{F})$. Then there exists a $z'\in\mathcal{F}$ such that $\|z_{\mathcal{N},\mathcal{M}}- z'\| \le \mathbb{D}(\mathcal{F}_{\mathcal{N},\mathcal{M}},\mathcal{F})$. From the Lipschitz continuity of $f$, we have
\begin{flalign*}
v=&f(z)\le f(z')
\le |f(z')-f(z_{\mathcal{N},\mathcal{M}})|+f(z_{\mathcal{N},\mathcal{M}}) \\ \le & L_f\|z'-z_{\mathcal{N},\mathcal{M}} \|+f(z_{\mathcal{N},\mathcal{M}})
\le  L_f\mathbb{D}(\mathcal{F}_{\mathcal{N},\mathcal{M}},\mathcal{F}) + v_{\mathcal{N},\mathcal{M}}.
\end{flalign*}
Exchanging the roles of $z_{\mathcal{N},\mathcal{M}}$ and $z$, we obtain $v_{\mathcal{N},\mathcal{M}}\le L_f\mathbb{D}(\mathcal{F},\mathcal{F}_{\mathcal{N},\mathcal{M}})+ v$. Applying conclusion (i), for any $\mathcal{N}$ and $\mathcal{M}$, we have
$$
|v_{\mathcal{N},\mathcal{M}}-v| \le L_f \mathbb{H}(\mathcal{F}_{\mathcal{N},\mathcal{M}},\mathcal{F})
\le \frac{2D_Z L_f}{\theta}(L\alpha_{\mathcal{N}}+\gamma_{\mathcal{M}}).
$$
}

{
Now, we show \eqref{equ:0927-1}. Let $z\in \mathcal{S}$ and $z_{\mathcal{N},\mathcal{M}} \in \mathcal{S}_{\mathcal{N},\mathcal{M}}$. Denote by $\Pi_A(x)\in \arginf_{s\in A}\|x-s\|$ the projection of a point $x$ on a set $A$. If problem \eqref{shenme2} satisfies the second order growth condition, then
\begin{flalign*}
 f(z_{\mathcal{N},\mathcal{M}})-f(\Pi_{\mathcal{F}}(z_{\mathcal{N},\mathcal{M}}))
 &=f(z_{\mathcal{N},\mathcal{M}})-f(z)-\big(f(\Pi_{\mathcal{F}}(z_{\mathcal{N},\mathcal{M}}))-f(z)\big) \\
 &\le f(\Pi_{\mathcal{F}_{\mathcal{N},\mathcal{M}}}(z))-f(z)-\rho d(\Pi_{\mathcal{F}}(z_{\mathcal{N},\mathcal{M}}),\mathcal{S})^2.
\end{flalign*}
Since $f$ is Lipschitz continuous, we have
$$
d(\Pi_{\mathcal{F}}(z_{\mathcal{N},\mathcal{M}}),\mathcal{S})\le \sqrt{(L_f/\rho)(\|\Pi_{\mathcal{F}_{\mathcal{N},\mathcal{M}}}(z) -z \|+\|\Pi_{\mathcal{F}}(z_{\mathcal{N},\mathcal{M}})-z_{\mathcal{N},\mathcal{M}} \|)}.
$$
By triangle inequality and the definition of projection $\Pi$, we get
\begin{flalign*}
& d(z_{\mathcal{N},\mathcal{M}}, \mathcal{S})
 \le \|z_{\mathcal{N},\mathcal{M}}-\Pi_{\mathcal{F}}(z_{\mathcal{N},\mathcal{M}}) \|+d(\Pi_{\mathcal{F}}(z_{\mathcal{N},\mathcal{M}}),\mathcal{S}) \\
 \le~ & \|z_{\mathcal{N},\mathcal{M}}-\Pi_{\mathcal{F}}(z_{\mathcal{N},\mathcal{M}}) \|+\sqrt{(L_f/\rho)\big(\|\Pi_{\mathcal{F}_{\mathcal{N},\mathcal{M}}}(z) -z \|+\|\Pi_{\mathcal{F}}(z_{\mathcal{N},\mathcal{M}})-z_{\mathcal{N},\mathcal{M}} \| \big)} \\
 = ~ & d(z_{\mathcal{N},\mathcal{M}},\mathcal{F})+\sqrt{(L_f/\rho)\big(d(z,\mathcal{F}_{\mathcal{N},\mathcal{M}}
)+d(z_{\mathcal{N},\mathcal{M}},\mathcal{F})\big)} \\
\le ~& \mathbb{H}(\mathcal{F}_{\mathcal{N},\mathcal{M}},\mathcal{F})
+\sqrt{(2L_f/\rho)\mathbb{H}(\mathcal{F}_{\mathcal{N},\mathcal{M}},\mathcal{F})}
\le  \big(1+\sqrt{{2L_f}/{\rho}} \big)\sqrt{\mathbb{H}(\mathcal{F}_{\mathcal{N},\mathcal{M}},\mathcal{F})},
\end{flalign*}
where the last inequality holds as $\mathbb{H}(\mathcal{F}_{\mathcal{N},\mathcal{M}},\mathcal{F})\le \sqrt{\mathbb{H}(\mathcal{F}_{\mathcal{N},\mathcal{M}},\mathcal{F})}$ for sufficiently large $\mathcal{N}$ and $\mathcal{M}$. Since $z_{\mathcal{N},\mathcal{M}}$ is arbitrarily chosen from $\mathcal{S}_{\mathcal{N},\mathcal{M}}$, then \eqref{equ:0927-1} follows from (i).}
\end{proof}

Theorem \ref{the:0927-1} establishes the quantitative convergency for problem \eqref{0925} in the sense of the feasible set, the optimal value and the optimal solution set. For the case that the support set $\Xi$ is finite, the lower bound approximation \eqref{0925} is tight.

\begin{corollary}
If $\Xi$ is a finite set, $\Xi_{\mathcal{N}}=\Xi$, 
and $\varGamma_{\mathcal{M}} = \{z_0^T\xi\mid\xi \in \Xi \}$, the optimal values of problem \eqref{shenme2} and problem \eqref{0925} are equal.
\end{corollary}
\begin{proof}
The conclusion follows from \cite[Proposition 3.2]{D1}.
\end{proof}

\subsection{\texorpdfstring{Tractability of the lower bound approximation problem \eqref{0925}}{ }}\label{subsec:3.2}
{Recall that $\Xi_{\mathcal{N}} =\{\bar{\xi_j}\}_{j=1}^{\mathcal{N}}$ and $\varGamma_{\mathcal{M}}=\{\eta_k\}_{k=1}^{\mathcal{M}}$, problem \eqref{0925} can be rewritten as
\begin{flalign*}
&
&\min \limits_{z \in Z}  \quad &  f(z)& \\
& &\quad\text{s.t.} \,\quad & \sup_{P \in \mathcal{Q}_{\mathcal{N}}}\mathbb{E}_{P}[(\eta_k-z^T\xi)_{+}-(\eta_k-z_0^T \xi)_{+}] \le 0,~k=1,\cdots,\mathcal{M}. &
\end{flalign*}
}%
Then by Lemma \ref{lemma}, its optimal value is equal to that of
\begin{equation}\label{flnew}
    \begin{array}{cl}
\min \limits_{z \in Z,\lambda \in \mathbb{R}_+^{\mathcal{M}}}  & f(z) \\
 \text{s.t.}  & \lambda_k \epsilon- \frac{1}{N}\sum\limits_{i=1}^{N}\min\limits_{1\le j\le \mathcal{N}} \left[\lambda_k \|\bar{\xi_j}-\widehat{\xi}_i\|-(\eta_k-z^T\bar{\xi_j})_{+}+(\eta_k-z_0^T \bar{\xi_j})_{+} \right] \le 0,  \\
& \qquad\quad\quad\qquad\qquad\qquad\qquad\qquad~\qquad\qquad\qquad~\qquad\qquad\qquad~\quad\ ~k=1,\cdots,\mathcal{M}.
    \end{array}
\end{equation}
By introducing auxiliary variables $\beta_{ik},~i=1,\cdots,N,k=1,\cdots,\mathcal{M}$,
problem \eqref{flnew} can be reformulated as
\begin{subequations}\label{problem3-8}
\begin{eqnarray}
&\min\limits_{z ,\lambda ,  \beta}   & f(z) \label{12}\\
& \text{s.t.}
&\lambda_k \epsilon- \frac{1}{N}\sum_{i=1}^{N}\beta_{ik} \le 0, ~ k=1,\cdots,\mathcal{M},\label{32add}\\
&& \beta_{ik} \le \lambda_k \|\bar{\xi_j}-\widehat{\xi}_i\|-(\eta_k-z^T\bar{\xi_j})_{+}
 +(\eta_k-z_0^T \bar{\xi_j})_{+} ,  \label{cococa}\\
&& \ \ \qquad\ i=1, \cdots, N,~ j=1,\cdots,\mathcal{N},~ k=1,\cdots,\mathcal{M}, \nonumber\\
&& z \in Z,\lambda \in \mathbb{R}_+^{\mathcal{M}}, \beta \in \mathbb{R}^{N \times \mathcal{M}}.\label{32}
\end{eqnarray}
\end{subequations}
Problem \eqref{problem3-8} is equivalent to problem \eqref{0925}, and thus is a lower bound approximation of problem ($P_{SSD}$).

By introducing auxiliary variables $s_{jk},~j=1,\cdots,\mathcal{N},k=1,\cdots,\mathcal{M}$, to handle $(\eta_k-z^T\bar{\xi_j})_+$ (refer to \cite{D1} (3.10)-(3.12)), we have a linear programming reformulation of problem \eqref{problem3-8}
\begin{flalign*}
&
&\min \limits_{z ,\lambda,  \beta ,s} \quad&  f(z) &\\
& &\text{s.t.} \ \quad &\lambda_k \epsilon- \sum_{i=1}^{N}\frac{1}{N}\beta_{ik} \le 0, ~k=1,\cdots,\mathcal{M}, & \\
(P_{SSD-L})& & &\beta_{ik}+s_{jk} \le \lambda_k \|\bar{\xi_j}-\widehat{\xi}_i\|
 +(\eta_k-z_0^T \bar{\xi_j})_{+} , &\\
& & & \qquad\ \ i=1, \cdots, N,~ j=1,\cdots,\mathcal{N},~ k=1,\cdots,\mathcal{M}, &\\
& & &s_{jk}\ge \eta_k-z^T\bar{\xi_j},~ j=1,\cdots,\mathcal{N},~ k=1,\cdots,\mathcal{M},& \\
&&&z \in Z,\lambda \in \mathbb{R}_+^{\mathcal{M}},  \beta \in \mathbb{R}^{N \times \mathcal{M}},s \in\mathbb{R}_+^{\mathcal{N} \times \mathcal{M}}.&
\end{flalign*}
In fact, the dimension of $s$ is $\mathcal{N} \times \mathcal{M}$ and the number of constraints in problem ($P_{SSD-L}$) is $\mathcal{M}+ N \times \mathcal{N} \times \mathcal{M} + \mathcal{N} \times \mathcal{M}$. {Thus the size of problem ($P_{SSD-L}$) increases rapidly with the increase of the sample sizes $\mathcal{N}$ and $\mathcal{M}$.}

\begin{algorithm}[ht]
\caption{Cutting-plane Method}
\label{alg:11}
\begin{algorithmic}
\STATE \textbf{Start from} ${\iota}=1$ and $\mathcal{J}_1^{\iota}=\mathcal{J}_2^{\iota}=\emptyset$.
\WHILE{${\iota} \ge 1$}
\STATE Solve the approximate problem:
\begin{flalign}
&
&\min \limits_{z ,\lambda,  \beta ,s} \ &  f(z) &\nonumber\\
& &\text{s.t.} \ \
&\beta_{ik}+s_{jk} \le \lambda_k \|\bar{\xi_j}-\widehat{\xi}_i\|
 +(\eta_k-z_0^T \bar{\xi_j})_{+} , i=1, \cdots, N, j \in \mathcal{J}_1^{{\iota}},k \in \mathcal{J}_2^{{\iota}},& \nonumber\\
& & &s_{jk}\ge \eta_k-z^T\bar{\xi_j}, j\in \mathcal{J}_1^{{\iota}}, k \in \mathcal{J}_2^{l{\iota}},&  \label{approximate}\\
&&& \eqref{32add},~z \in Z,\lambda \in \mathbb{R}_+^{\mathcal{M}},  \beta \in \mathbb{R}^{N \times \mathcal{M}},s \in\mathbb{R}_+^{\mathcal{N} \times \mathcal{M}}.& \nonumber
\end{flalign}
\STATE Let $(z^{{\iota}},\lambda^{{\iota}}, \beta^{{\iota}},s^{{\iota}})$ denote the optimal solution of problem \eqref{approximate}.
\STATE Calculate
\begin{equation*}
\begin{split}
  &\delta^{{\iota}}:= \max \limits_{i \in \{1,\cdots,N\}, j\in \{1,\cdots,\mathcal{N}\},\atop k\in \{1,\cdots,\mathcal{M}\}}
  \left\{
   \beta_{ik}^{{\iota}}
  - \lambda_k^{{\iota}} \|\bar{\xi_j}-\widehat{\xi}_i\|
  + (\eta_k-(z^{{\iota}})^T\bar{\xi_j})_{+}
  - (\eta_k-z_0^T \bar{\xi_j})_{+} \right\}.
  \end{split}
\end{equation*}
\IF{$\delta^{{\iota}} \le 0$}
\STATE \text{Stop}.\\
\ELSE{}
\STATE Determine
\begin{equation*}
\begin{split}
  &(i^{\iota},j^{\iota},k^{\iota})\in \\
  &\argmax \limits_{i \in \{1,\cdots,N\},j\in \{1,\cdots,\mathcal{N}\},\atop k\in \{1,\cdots,\mathcal{M}\}}
  \left\{
  \beta_{ik}^{{\iota}}
  - \lambda_k^{{\iota}} \|\bar{\xi_j}-\widehat{\xi}_i\|
  + (\eta_k-(z^{{\iota}})^T\bar{\xi_j})_{+}
  - (\eta_k-z_0^T \bar{\xi_j})_{+}
  \right\}.
  \end{split}
\end{equation*}
\STATE Let $\mathcal{J}_1^{{\iota}+1} = \mathcal{J}_1^{{\iota}} \cup j^{\iota}$, $\mathcal{J}_2^{{\iota}+1} = \mathcal{J}_2^{{\iota}} \cup k^{\iota}$ and ${\iota} \leftarrow {\iota}+1$.
\ENDIF
\ENDWHILE
\end{algorithmic}
\end{algorithm}

In order to numerically solve problem \eqref{problem3-8} for large $\mathcal{N},\mathcal{M}$, we propose a cutting-plane method, see Algorithm \ref{alg:11}.
{At each iteration of the cutting-plane method, we solve problem \eqref{approximate}, a relaxation of problem \eqref{problem3-8}.} After solving \eqref{approximate}, we check whether all the constraints in \eqref{cococa} are satisfied or not. If all the constraints in \eqref{cococa} are satisfied, then the optimal solution we find for problem \eqref{approximate} is also optimal for problem \eqref{problem3-8}. {Otherwise, we add the violated constraint to the approximate problem \eqref{approximate} at the next iteration.}

\begin{proposition}
  Algorithm \ref{alg:11} stops at the optimal value and optimal solution of problem \eqref{problem3-8} within finite steps.
\end{proposition}
\begin{proof}
It is easy to see that $\mathcal{J}_1^{\iota} \subsetneqq \mathcal{J}_1^{{\iota}+1}$ or $\mathcal{J}_2^{{\iota}} \subsetneqq \mathcal{J}_2^{{\iota}+1}$.
As the possible number of constraints that can be added is finite, Algorithm \ref{alg:11} must stop at the optimal value and optimal solution of problem \eqref{problem3-8} within finite steps.
\end{proof}

To conclude this section, we develop a lower bound approximation  \eqref{0925} for the distributionally robust {SSD} constrained optimization problem ($P_{SSD}$). Problem \eqref{0925} can be reformulated as problem \eqref{problem3-8}, which can be easily solved using linear programming formulation ($P_{SSD-L}$) or by Algorithm \ref{alg:11}.

\section{Upper bound approximation of distributionally robust SSD constrained problem}\label{3.2}

Notice that problem \eqref{shenme2} can be rewritten as
\begin{equation}\label{shenme3}\begin{array}{cl}
\min \limits_{z \in Z}  \quad &  f(z) \\
\quad\text{s.t.} \ \ \quad &\sup \limits_{P \in \mathcal{Q}}\sup \limits_{\eta \in \mathcal{R}} \mathbb{E}_{P}[(\eta-z^T\xi)_{+}-(\eta-z_0^T \xi)_{+}] \le 0.
\end{array}\end{equation}
{If we exchange the order of operators $\sup_{\eta \in \mathcal{R}}$ and $\mathbb{E}_{P}$ in problem \eqref{shenme3}, we obtain an upper bound approximation for problem \eqref{shenme3}. However, such an upper bound approximation might be loose {or even infeasible} since 
the gap
\begin{equation}\label{equ:gap}
 \mathbb{E}_{P}\left[ \sup_{\eta \in \mathcal{R}}\{(\eta-z^T\xi)_{+}-(\eta-z_0^T \xi)_{+}\} \right]
-
\sup_{\eta \in \mathcal{R}} \mathbb{E}_{P}[(\eta-z^T\xi)_{+}-(\eta-z_0^T \xi)_{+}]
\end{equation}
might be large. This is because we determine an $\eta$ for all possible $\xi$'s in the latter supremum in \eqref{equ:gap}, while we determine an $\eta$ for each realization of $\xi$ in the former supremum in \eqref{equ:gap}. The larger the range $\mathcal{R}$ of $\eta$, the larger the gap in \eqref{equ:gap}. As an extreme case, when $\mathcal{R}$ reduces to a singleton, the gap \eqref{equ:gap} becomes $0$. This observation motivates us to divide $\mathcal{R}$ into small sub-intervals, and exchange the order of the expectation operator and the supremum over each sub-interval, which provides an upper bound approximation of the sub-problem 
in the sub-interval.
Summing all sub-problems in all sub-intervals, we obtain an improved upper bound approximation of problem \eqref{shenme3}. We name such a bounding method a split-and-dual framework.}

In detail, we divide $\mathcal{R}=[\mathcal{R}_{{\min}},\mathcal{R}_{{\max}}]$ into $\mathcal{K}$ intervals with disjoint interiors, $[\uline{\eta}_{k},\bar{\eta}_{k}]$, $k=1,\cdots,\mathcal{K}$, where the boundary points of the intervals are specified by $\uline{\eta}_{k}=\mathcal{R}_{{\min}}+(k-1)\frac{\mathcal{R}_{{\max}}-\mathcal{R}_{{\min}}}{\mathcal{K}}$,
$\bar{\eta}_{k}=\mathcal{R}_{{\min}}+k\frac{\mathcal{R}_{{\max}}-\mathcal{R}_{{\min}}}{\mathcal{K}}$, $k=1,\cdots,\mathcal{K}$. Notice that problem \eqref{shenme3} can also be reformulated as
\begin{flalign*}
&
&\min \limits_{z \in Z}  \quad & f(z)  &\\
& &\quad\text{s.t.} \,\ \ & \max_{1\le k\le\mathcal{K}}\sup_{\eta \in [\uline{\eta}_{k},\bar{\eta}_{k}]}\sup_{P \in \mathcal{Q}}\mathbb{E}_{P}[(\eta-z^T\xi)_{+}-(\eta-z_0^T \xi)_{+}] \le 0, &
\end{flalign*}
or, equivalently,
\begin{subequations}\label{problem4-3}
\begin{eqnarray}
&\min \limits_{z \in Z}   & f(z)  \label{equ:fl1}\\
 &\quad\text{s.t.} \,\,\ & \sup_{P \in \mathcal{Q}}\sup_{\eta \in [\uline{\eta}_{k},\bar{\eta}_{k}]}\mathbb{E}_{P}[(\eta-z^T\xi)_{+}-(\eta-z_0^T \xi)_{+}] \le 0,~k=1,\cdots,\mathcal{K}.\label{equ:fl2}
\end{eqnarray}
\end{subequations}
Exchanging the order of operators $\sup_{\eta \in [\uline{\eta}_{k},\bar{\eta}_{k}]}$ and $\mathbb{E}_P$ in \eqref{equ:fl2}, we have the following approximation problem
\begin{equation}\label{222}\begin{array}{cl}
\min \limits_{z \in Z}\, \  & f(z) \\
\quad\text{s.t.}\ \ \ & \sup \limits_{P \in \mathcal{Q}}\mathbb{E}_{P}\Big[\sup \limits_{\eta \in [\uline{\eta}_{k},\bar{\eta}_{k}] }\{(\eta-z^T\xi)_{+}-(\eta-z_0^T \xi)_{+}\}\Big] \le 0,~k=1,\cdots,\mathcal{K}.
\end{array}\end{equation}
The feasible set of problem \eqref{problem4-3} contains that of problem \eqref{222}.
Thus problem \eqref{222} provides an upper bound approximation for problem \eqref{problem4-3}.

\subsection{Quantitative analysis of the upper approximation}
In what follows, we show that when the interval number $\mathcal{K}$ goes to infinity,
the optimal value of problem \eqref{222} converges to that of \eqref{problem4-3}. To this end, we first prove the convergence from
$$
g(z,\mathcal{K}):=\max_{1\le k \le \mathcal{K}} \sup_{P \in \mathcal{Q}}\mathbb{E}_{P}
         \Big[\sup_{\eta \in [\uline{\eta}_{k},\bar{\eta}_{k}] }\left\{(\eta-z^T\xi)_{+}-(\eta-z_0^T \xi)_{+}\right\}\Big]
$$
to
\begin{equation*}
\begin{split}
g(z):
&=\sup_{P \in \mathcal{Q}}\sup_{\eta \in \mathcal{R}} \mathbb{E}_{P}\left[(\eta-z^T\xi)_{+}-(\eta-z_0^T \xi)_{+}\right] \\
&= \max_{1\le k \le \mathcal{K}} \sup_{P \in \mathcal{Q}}\sup_{\eta \in [\uline{\eta}_{k},\bar{\eta}_{k}]}\mathbb{E}_{P}\left[(\eta-z^T\xi)_{+}-(\eta-z_0^T \xi)_{+}\right].
\end{split}
\end{equation*}
As $\cup_{k=1,\ldots,\mathcal{K}}[\uline{\eta}_{k},\bar{\eta}_{k}]=\mathcal{R}$, the function $g(z)$ does not depend on the splitting of $\mathcal{R}$.
%
\begin{proposition}\label{pro:lip}
Given Assumption \ref{ass:Xi}, for any positive integer $\mathcal{K}$, $g(\cdot,\mathcal{K})$ and $g(\cdot)$ are Lipschitz continuous {with modulus $\mathcal{C}=\sup_{P \in \mathcal{Q}}\mathbb{E}_P[\|\xi\|]< \infty$}.
\end{proposition}
\begin{proof}
  It follows from \cite[page 164]{refnew}.
\end{proof}

Next, we prove that $g(z,\mathcal{K})$ converges to $g(z)$ when $\mathcal{K}$ goes to infinity.
\begin{proposition}\label{propp}
Given Assumption \ref{ass:Xi}, we have that
$$ g(z,\mathcal{K})-g(z) \leq 2 \frac{\mathcal{R}_{{\max}}-\mathcal{R}_{{\min}}}{\mathcal{K}}, $$
and $\lim_{\mathcal{K} \to \infty}g(z,\mathcal{K}) {=}g(z)$, uniformly with respect to $z \in Z$.
\end{proposition}
\begin{proof}
  Denote
  $$
  \eta^{*}_{k}(\omega)\in\argsup_{\eta \in [\uline{\eta}_{k},\bar{\eta}_{k}] }\left \{(\eta-z^T\xi(\omega))_{+}-(\eta-z_0^T \xi(\omega))_{+}\right\},~\omega \in \Omega,~k=1,\cdots,\mathcal{K}
  $$
  Denote
  $$
  \eta^{**}_{k} \in \argsup_{\eta \in [\uline{\eta}_{k},\bar{\eta}_{k}]}\mathbb{E}_{P}\left[(\eta-z^T\xi)_{+}-(\eta-z_0^T \xi)_{+}\right],~k=1,\cdots,\mathcal{K}.
  $$
Notice that $\eta^{*}_{k}$ is a random variable, while
  $\eta^{**}_{k}$ is a real number.
 Since $\eta^{*}_{k}(\omega)$ and $\eta^{**}_{k}$ take values in the same interval $[\uline{\eta}_{k},\bar{\eta}_{k}]$, for any $\omega \in \Omega$, we have
  $|\eta^{*}_{k}(\omega)-\eta^{**}_{k}| \le \bar{\eta}_{k}-\uline{\eta}_{k}=\frac{\mathcal{R}_{{\max}}-\mathcal{R}_{{\min}}}{\mathcal{K}}$.
  Then we obtain
  \begin{flalign}
   &   g(z,\mathcal{K})-g(z) \label{equ:bbb}
  \le  \max_{1\le k \le \mathcal{K}} \sup_{P \in \mathcal{Q}}
 \left\{ \mathbb{E}_{P}
         \Big[\sup_{\eta \in [\uline{\eta}_{k},\bar{\eta}_{k}] }\left \{(\eta-z^T\xi)_{+}-(\eta-z_0^T \xi)_{+}\right\}\Big] \right. \nonumber \\
         & \qquad\qquad\qquad\ \ \  -   \left. \sup_{\eta \in [\uline{\eta}_{k},\bar{\eta}_{k}]}\mathbb{E}_{P}\left[(\eta-z^T\xi)_{+}-(\eta-z_0^T \xi)_{+}\right]
         \right\} \nonumber \\
   \le & \max_{1\le k \le \mathcal{K}} \sup_{P \in \mathcal{Q}}
          \mathbb{E}_{P} \Big| \left [(\eta^*_{k}-z^T\xi)_{+}-(\eta^*_{k}-z_0^T \xi)_{+}\right]
          -
          \left[(\eta^{**}_{k}-z^T\xi)_{+}-(\eta^{**}_{k}-z_0^T \xi)_{+}\right] \Big| \nonumber\\
   \le & ~ 2 \frac{\mathcal{R}_{{\max}}-\mathcal{R}_{{\min}}}{\mathcal{K}}, \nonumber
  \end{flalign}
  where the last inequality is due to the Lipschitz continuity of the positive part function $(\cdot)_{+}$.
  Then the conclusion immediately follows.
\end{proof}

{Proposition \ref{propp} shows that to control the approximation error of the constraint function, the interval number $\mathcal{K}$ should be large enough when the range $\mathcal{R}$ is large.}

We denote the feasible sets of problem \eqref{problem4-3} and problem \eqref{222} by $\mathcal{F}$ and $\mathcal{F}_{\mathcal{K}}$, the optimal solution sets by $\mathcal{S}$ and $\mathcal{S}_{\mathcal{K}}$, and the optimal values by $v$ and $v_{\mathcal{K}}$, respectively. It is clear that $\mathcal{F}_{\mathcal{K}} \subset \mathcal{F},~\forall \mathcal{K}$. To derive the convergence from $v_{\mathcal{K}}$ to $v$, as well as the quantitative approximation error estimation, we need some constraint qualification, e.g., Mangasarian Fromovitz constraint qualification (MFCQ) \cite{mangasarian1967fritz}.
However, classical MFCQ works only in the differentiable case (e.g., \cite{LPL20}), while function $g(\cdot,\mathcal{K})$ here is non-smooth.
We can find that $g(\cdot,\mathcal{K})$ and $g(\cdot)$ are continuous and convex, and thus subdifferentiable everywhere. Therefore, it is reasonable for us to extend MFCQ to the subdifferentiable case.

\begin{definition}\label{def:11}{\rm (ND-MFCQ)}
  Let $F(t) := \{x \in \mathbb{R}^n \mid g_j(x, t) \le 0, j \in J\}$ with subdifferentiable $g_j$, here $t$ is the parameter in the constraints. Given $\bar{t}$ and $\bar{x}\in F(\bar{t})$, if there exist some vector $\theta$ and real constants $\sigma<0$, $\alpha_1 >0$, $\alpha_2 >0$ such that
$$
\langle \varsigma, \theta \rangle \le \sigma<0,~ \forall \varsigma \in \partial g_j({x},{t}), \forall x:\|x-\bar{x}\| \le \alpha_1  ,
  \forall t: \|t-\bar{t}\| \le \alpha_2,\forall j \in J_0(\bar{x},\bar{t}),
$$
where $J_0(\bar{x},\bar{t}):=\{j\in J \mid g_j(\bar{x},\bar{t})=0\}$, then we say that non-differentiable MFCQ (ND-MFCQ) holds at $(\bar{x},\bar{t} )$ with $\theta, \sigma, \alpha_1$ and $\alpha_2$,
\end{definition}

We notice that the MFCQ condition under nonsmooth cases has been discussed in some literature, such as \cite[Page 14]{doi}. {In} fact, the ND-MFCQ in Definition \ref{def:11} is more strict than that in \cite{doi}.
We require ND-MFCQ hold uniformly in a neighborhood of $(\bar{x},\bar{t})$ with the same vector $\theta$, while \cite{doi} only requires ND-MFCQ hold at the point $\bar{x}$.
ND-MFCQ 
is an extension of MFCQ \cite{mangasarian1967fritz}, and is equivalent to MFCQ if the constraint functions are differentiable.

In this paper, we only have one constraint and thus $J=\{1\}$. Our decision variable $z$ corresponds to $x$
and our parameter $\frac{1}{\mathcal{K}}$ corresponds to $t$ in Definition \ref{def:11}.
To arrive at the convergence result, we also require the following  assumption.

\begin{assumption}\label{ass:new}
The optimal solution set of problem \eqref{222} with $\mathcal{K}=1$, denoted by  $\mathcal{S}_{1}$, is nonempty.
\end{assumption}

From \eqref{equ:fl2} and the constraints in \eqref{222}, we have that $\mathcal{S} \supset \mathcal{S}_{\mathcal{K}} \supset \mathcal{S}_1$ for any $\mathcal{K}$. Therefore, if $\mathcal{S}_1$ is nonempty, then $\mathcal{S}$ and $\mathcal{S}_{\mathcal{K}},\forall \mathcal{K}$, are also nonempty.

\begin{theorem}\label{pro:cite}
Given Assumptions \ref{ass:Xi} and \ref{ass:new}. For some $z^* \in \mathcal{S}$, assume that ND-MFCQ holds at $(z^*,0)$ with $\theta$, $\sigma$, $\alpha_1$, and $\alpha_2$ as is defined in Definition \ref{def:11}.
If the objective function $f$ is Lipschitz continuous with modulus $L_f$, then for $\mathcal{K} \ge \max\left\{\frac{1}{\alpha_2}, \frac{2}{|\sigma|} \frac{\mathcal{R}_{\rm{max}}-\mathcal{R}_{\rm{min}}}{\alpha_1} \|\theta\|,-2 \frac{\mathcal{R}_{\rm{max}}-\mathcal{R}_{\rm{min}}}{g(z^*)}\left(\mathcal{C}\frac{\|\theta\|}{|\sigma|}+1 \right)\right\},$ we have that
  $$|v_{\mathcal{K}} -v| \le L_f  \frac{2\|\theta\|}{|\sigma|} \frac{\mathcal{R}_{\rm{max}}-\mathcal{R}_{\rm{min}}}{\mathcal{K}},$$
and $\lim_{\mathcal{K} \to \infty} v_{\mathcal{K}}=v$.
\end{theorem}

\begin{proof}
  For any $\mathcal{K}$, let $z_\mathcal{K}=z^* - \frac{2}{\sigma} \frac{\mathcal{R}_{{\max}}-\mathcal{R}_{{\min}}}{\mathcal{K}} \theta.$
  Since $z^* \in \mathcal{S}$, then obviously $g(z^*) \le 0$.

  Firstly, we claim that $z_\mathcal{K} \in \mathcal{F}_{\mathcal{K}}$. To prove this, we examine two cases on whether the constraint $g(z) \le 0$ is active at $z^*$.

  $\bullet$  Case 1: $g(z^*)=0$.
  By Proposition \ref{propp}, we have
    \begin{flalign}
   & g(z_\mathcal{K},\mathcal{K})
     = g(z_\mathcal{K},\mathcal{K}) - g(z^*, \mathcal{K}) + g(z^*, \mathcal{K}) - g(z^*) + g(z^*) \nonumber\\
     = &  g(z_\mathcal{K},\mathcal{K}) - g(z^*, \mathcal{K}) + g(z^*, \mathcal{K})- g(z^*)
     \le g(z_\mathcal{K},\mathcal{K}) - g(z^*, \mathcal{K}) + 2 \frac{\mathcal{R}_{{\max}}-\mathcal{R}_{{\min}}}{\mathcal{K}}. \nonumber
   \end{flalign}
  By the extended mean-value theorem \cite[Theorem 10.48]{Rockafellar2009Variational}, for some $\tau \in (0,1)$ and the corresponding point $z^{\tau}_{\mathcal{K}} = (1-\tau)z_{\mathcal{K}} + \tau z^*$, there is a vector $\varsigma \in \partial g(z^{\tau}_{\mathcal{K}},\mathcal{K})$ satisfying
   \begin{flalign}
    g(z_\mathcal{K},\mathcal{K})
    \le g(z_\mathcal{K},\mathcal{K}) - g(z^*, \mathcal{K}) + 2 \frac{\mathcal{R}_{{\max}}-\mathcal{R}_{{\min}}}{\mathcal{K}}
    = \langle \varsigma, z_\mathcal{K}-z^* \rangle + 2 \frac{\mathcal{R}_{{\max}}-\mathcal{R}_{{\min}}}{\mathcal{K}}. \nonumber
  \end{flalign}
  For any $\mathcal{K}\ge \frac{2}{|\sigma|} \frac{\mathcal{R}_{{\max}}-\mathcal{R}_{{\min}}}{\alpha_1} \|\theta\|$, $z^{\tau}_{\mathcal{K}}$ is in the $\alpha_1$-neighborhood of $z^*$,
  which can be seen from
  \begin{flalign*}
  \|z^{\tau}_{\mathcal{K}} - z^*\|
  = \|(1-\tau)(z_{\mathcal{K}}-z^*)\|
  \le \|z_{\mathcal{K}}-z^*\|
  = \frac{2}{|\sigma|} \frac{\mathcal{R}_{{\max}}-\mathcal{R}_{{\min}}}{\mathcal{K}} \|\theta\|
  \le \alpha_1.
  \end{flalign*}
  Therefore, by ND-MFCQ at $(z^*,0)$, we have $\langle \varsigma,\theta\rangle\leq\sigma $ for $\varsigma\in \partial g(z^{\tau}_{\mathcal{K}},\mathcal{K})$ and $\mathcal{K} \ge \max\left\{\frac{2}{|\sigma|} \frac{\mathcal{R}_{{\max}}-\mathcal{R}_{{\min}}}{\alpha_1} \|\theta\|,\frac{1}{\alpha_2}\right\}$.
  Then we have
  \begin{flalign*}
    g(z_\mathcal{K},\mathcal{K})
    & \le \langle \varsigma, z_\mathcal{K}-z^* \rangle + 2 \frac{\mathcal{R}_{{\max}}-\mathcal{R}_{{\min}}}{\mathcal{K}}
     = - \frac{2}{\sigma} \frac{\mathcal{R}_{{\max}}-\mathcal{R}_{{\min}}}{\mathcal{K}}
    \langle \varsigma,  \theta \rangle + 2 \frac{\mathcal{R}_{{\max}}-\mathcal{R}_{{\min}}}{\mathcal{K}} \\
    & \le - \frac{2}{\sigma} \frac{\mathcal{R}_{{\max}}-\mathcal{R}_{{\min}}}{\mathcal{K}} \sigma + 2 \frac{\mathcal{R}_{{\max}}-\mathcal{R}_{{\min}}}{\mathcal{K}}=0. \\
  \end{flalign*}
  Then $z_\mathcal{K} \in \mathcal{F}_{\mathcal{K}}$.

  $\bullet$ Case 2: $g(z^*) <0$. Let $\delta: = -g(z^*)>0$.
  If $\mathcal{K}\ge 2 \frac{\mathcal{R}_{{\max}}-\mathcal{R}_{{\min}}}{\delta}\left(\mathcal{C}\frac{\|\theta\|}{|\sigma|}+1 \right)$, then we obtain from Propositions \ref{pro:lip} and \ref{propp} that
  \begin{flalign*}
   & |g(z_\mathcal{K},\mathcal{K}) -g(z^*)|
     \le
    |g(z_\mathcal{K},\mathcal{K}) -g(z^*,\mathcal{K}) | +|g(z^*,\mathcal{K})-g(z^*) | \\
    \le
  &  \mathcal{C} \|z_\mathcal{K}-z^* \| +2 \frac{\mathcal{R}_{{\max}}-\mathcal{R}_{{\min}}}{\mathcal{K}}
  =2 \frac{\mathcal{R}_{{\max}}-\mathcal{R}_{{\min}}}{\mathcal{K}} \left(\mathcal{C}\frac{\|\theta\|}{|\sigma|}+1 \right) \le \delta.
  \end{flalign*}
  This indicates that
  $  g(z_\mathcal{K},\mathcal{K}) \le g(z^*)+\delta =0.  $
Thus $z_\mathcal{K} \in \mathcal{F}_{\mathcal{K}}$ under this case.

  Next, we estimate the approximation error $|v_{\mathcal{K}}-v|$ of the optimal values. Since $\mathcal{F}_{\mathcal{K}} \subset \mathcal{F}$, then $v_{\mathcal{K}} \ge v$. By Assumption \ref{ass:new}, we can choose $z^*_{\mathcal{K}} \in \mathcal{S}_{\mathcal{K}}$. 
  Then we have
  \begin{flalign*}
    0 & \le v_{\mathcal{K}} -v
     = f(z^*_{\mathcal{K}}) -f (z^*) \le f(z_{\mathcal{K}}) -f(z^*)
     \\
     &\le L_f \|z_{\mathcal{K}} -z^* \| =L_f  \frac{2\|\theta\|}{|\sigma|} \frac{\mathcal{R}_{\max}-\mathcal{R}_{{\min}}}{\mathcal{K}}.
  \end{flalign*}
  The conclusions follow immediately.
\end{proof}

Proposition \ref{pro:cite} quantitatively estimates the approximation error between the optimal values of problem \eqref{222} and problem \eqref{problem4-3}.

\subsection{\texorpdfstring{Reformulation of problem \eqref{222}}{ }}
By applying Lemma \ref{lemma} to each supremum problem w.r.t. $P$ for $k=1,\ldots,\mathcal{K}$, we have a reformulation of problem \eqref{222}
\begin{subequations}\label{problem4-6}
\begin{eqnarray}
&\quad\ \min_{z \in Z, \lambda \in \mathbb{R}_{+}^{
 \mathcal{K}}} & f(z) \label{obj}\\
&\quad\ \text{s.t.}  &
 \lambda_{k} \epsilon +\frac{1}{N}\sum_{i=1}^{N}\sup_{\xi \in \Xi} \bigg\{ \label{constraint} \\
 &&\sup_{\eta \in [\uline{\eta}_{k},\bar{\eta}_{k}]}\{(\eta-z^T\xi)_+-(\eta-z_0^T\xi)_+\}-\lambda_{k} \|\xi-\widehat{\xi}_i\| \bigg\}\le 0,~k=1,\cdots,\mathcal{K}.\nonumber
\end{eqnarray}
\end{subequations}
To simplify the notation, we write \eqref{constraint} as
\begin{equation}\label{equ:sss}
  \lambda_{k} \epsilon +\frac{1}{N}\sum_{i=1}^{N} V^{ik}_{S} \le 0,~k=1,\cdots,\mathcal{K},
\end{equation}
where
\begin{equation*}
V^{ik}_{S}:=\sup \limits_{(\xi,\eta)\in \Xi \times [\uline{\eta}_{k},\bar{\eta}_{k}]} (\eta-z^T\xi)_+-(\eta-z_0^T\xi)_+-\lambda_{k} \|\xi-\widehat{\xi}_i\|,~i=1,\cdots,N,~k=1,\cdots,\mathcal{K}.
\end{equation*}

In what follows, we derive a reformulation for $V^{ik}_{S}$. According to Assumption \ref{ass:Xi}, $V^{ik}_{S}$ is equivalent to
\begin{equation}\label{equ:SOCP}
\begin{split}
\sup \limits_{\xi,\eta ,s ,m}  \quad & (\eta-z^T\xi)_{+}-s-\lambda_{k} m\\
\quad\text{s.t.} \,\quad   & s \ge \eta-z_0^T \xi,~C\xi \le d,\\
                         &  \eta \ge \uline{\eta}_{k},\eta \le \bar{\eta}_{k},
                         ~ \|\xi-\widehat{\xi}_i\|\le m,\\
                         & \xi \in \mathbb{R}^n,\eta \in \mathbb{R},s \in \mathbb{R}_+, m\in \mathbb{R}.
\end{split}
\end{equation}
Problem \eqref{equ:SOCP} is a non-convex optimization problem with a piecewise linear objective function with two pieces. Examining the two pieces of the objective function separately,
we can split problem \eqref{equ:SOCP} into two convex sub-problems:
\begin{flalign*}
& \quad V^{ik}_{S1}=  &\sup_{\xi ,\eta,s,m} \quad  \eta-z^T\xi-s-\lambda_{k} m\,& &
& \quad V^{ik}_{S2}=  &\sup_{\xi ,\eta,s,m} \quad -s-\lambda_{k} m \quad \ \\
& &\qquad\text{s.t.} \quad\  s\ge\eta-z_0^T\xi,\qquad\ \ \ &&
  &&\qquad\text{s.t.} \quad\ s \ge\eta-z_0^T\xi, \ \\
& &\eta-z^T \xi \ge 0,\ \ \  \qquad& \qquad & & &\eta-z^T\xi\le0,\ \\
&(P_{SSD-1}^{ik}) &C\xi \le d,\qquad \qquad\ \ \ & &
  &(P_{SSD-2}^{ik})  &C\xi \le d, \qquad \ \\
& & s \ge 0,\qquad\quad \qquad\ \ \ & & & & s \ge 0,\quad\qquad \ \\
& & \eta \ge \uline{\eta}_{k},\ \,\quad\quad \qquad\ \ \ & & & & \eta \ge \uline{\eta}_{k},\ \ \ \ \ \,\quad\\
& & \eta \le \bar{\eta}_{k}, \ \, \quad \quad\qquad \ \ \ & & & & \eta \le \bar{\eta}_{k}, \ \ \ \ \ \,\quad\\
& & \|\xi-\widehat{\xi}_i\|\le m.\quad\ \ \ \ \  & & & & \|\xi-\widehat{\xi}_i\|\le m.
\end{flalign*}
And we have
\begin{equation}\label{equ:V}
V^{ik}_{S}=\max\{V^{ik}_{S1},V^{ik}_{S2}\}.
\end{equation}

Using conic duality theory, we derive the dual problem of ($P_{SSD-1}^{ik}$) as follows:
\begin{flalign*}
     \tilde{V}_{S1}^{ik}& &=\inf_{\mu,\nu}\
    & d^T\nu-\widehat{\xi}_i^T (z -\mu_1 z_0 +\mu_2 z +C^T\nu) -\mu_3 \uline{\eta}_{k} +(1\!-\!\mu_1+\mu_2+\mu_3) \bar{\eta}_{k} &\\
 (D_{SSD-1}^{ik})
 & &\text{s.t.} \    &\mu_1 \le 1, ~1-\mu_1+\mu_2+\mu_3\ge 0, \\
  & &  & \|z -\mu_1 z_0+ \mu_2 z + C^T \nu\| \le \lambda_{k},\\
  &&& \mu \in \mathbb{R}_+^3,~\nu \in \mathbb{R}_+^l.&
\end{flalign*}
Likewise, the dual problem of ($P_{SSD-2}^i$) is
\begin{align*}
 \tilde{V}_{S2}^{ik}& & =  \inf_{\mu,\nu}\
  & d^T\nu-\widehat{\xi}_i^T (-\mu_1z_0-\mu_2z+C^T \nu)-\mu_3 \uline{\eta}_{k} +( -\mu_1-\mu_2+\mu_3) \bar{\eta}_{k} &\ \ \ \\
(D_{SSD-2}^{ik})\ & &  \text{s.t.} \
   & \mu_1\le 1,~ -\mu_1-\mu_2+\mu_3\ge 0, \\
  & &  &\|-\mu_1z_0-\mu_2z + C^T \nu\| \le \lambda_{k},\\
  && & \mu \in \mathbb{R}_+^3,~\nu \in \mathbb{R}_+^l. &
\end{align*}
By equation \eqref{equ:V} and the duality theory, we have that $V^{ik}_{S} \le \max \{\tilde{V}_{S1}^{{ik}},\tilde{V}_{S2}^{ik}\}$, $i=1,\cdots,N,~k=1,\cdots,\mathcal{K}$.
\begin{assumption}\label{ass:strictlyfeasible}
  For any $i=1,\cdots,N,~k=1,\cdots,\mathcal{K}$,
  problems \emph{(}$P_{SSD-1}^{ik}$\emph{)} and \emph{(}$P_{SSD-2}^{ik}$\emph{)} are strictly feasible.
\end{assumption}

Given Assumption \ref{ass:strictlyfeasible}, the strong duality holds. Thus, the duality gap between ${V}_{S1}^{{ik}}$  (resp. ${V}_{S2}^{{ik}}$) and $\tilde{V}_{S1}^{{ik}}$ (resp. $\tilde{V}_{S2}^{{ik}}$) is zero,
and $V^{ik}_{S} = \max \{\tilde{V}_{S1}^{ik},\tilde{V}_{S2}^{ik}\}$, $i=1,\cdots,N,~k=1,\cdots,\mathcal{K}$.
Introducing auxiliary variables $V^{ik},~i=1,\cdots,N,~k=1,\cdots,\mathcal{M} $, constraints
\begin{equation}\label{problem4-10}
    \begin{array}{l}
        \lambda_{k} \epsilon + \frac{1}{N}\sum \limits_{i=1}^{N}V^{ik} \le 0, k = 1,\cdots, \mathcal{K},\\ 
V^{ik} \ge \tilde{V}_{S1}^{ik},~i=1,\cdots,N,k = 1,\cdots, \mathcal{K},\\ 
V^{ik} \ge \tilde{V}_{S2}^{ik},~i=1,\cdots,N,k = 1,\cdots, \mathcal{K} 
    \end{array}
\end{equation}
are equivalent to constraints \eqref{equ:sss}.
Taking the formulations ($D_{SSD-1}^{ik}$) and ($D_{SSD-2}^{ik}$) of $\tilde{V}_{S1}^{{ik}}$ and $\tilde{V}_{S2}^{{ik}}$
into constraints \eqref{problem4-10} gives the following theorem.
\begin{theorem}\label{the:4}
Given Assumptions \ref{ass:Xi} and \ref{ass:strictlyfeasible},
the optimal value of the following optimization problem
\begin{footnotesize}
\begin{equation*}
\begin{split}
\min \ &\ \ f(z)   \\
{\rm s.t.} \ &\ \ \lambda_{k} \epsilon + \frac{1}{N}\sum_{i=1}^{N}V^{ik} \le 0,~ k=1, \cdots,\mathcal{K},\\
(P_{SSD-U}) & \left.
\begin{array}{l}
  \mu_{1}^{ik} \le 1,~\tilde{\mu}_{1}^{ik} \le 1,~1-\mu_1^{ik}+\mu_2^{ik}+\mu_3^{ik}\ge 0,~-\tilde{\mu}_1^{ik}-\tilde{\mu}_2^{ik}+\tilde{\mu}_3^{ik}\ge 0, \\
  V^{ik} \ge d^T\nu^{ik}\!-\widehat{\xi}_i^T (z\!-\mu_{1}^{ik}z_{0}\!+\mu_2^{ik}z\!+C^T \nu^{ik}) \!-\mu_{3}^{ik}\uline{\eta}_{k}\!+(1\!-\mu_1^{ik}\!+\mu_2^{ik}\!+\mu_3^{ik}) \bar{\eta}_{k}, \\
  V^{ik} \ge d^T\tilde{\nu}^{ik}-\widehat{\xi}_i^T (-\tilde{\mu}_1^{ik}z_0-\tilde{\mu}_2^{ik}z+C^T \tilde{\nu}^{ik})-\tilde{\mu}_3^{ik} \uline{\eta}_{k} +(-\tilde{\mu}_1^{ik}-\tilde{\mu}_2^{ik}+\tilde{\mu}_3^{ik}) \bar{\eta}_{k}, \\
  \|z -\mu_1^{ik} z_0+ \mu_2^{ik} z + C^T \nu^{ik}\| \le \lambda_{k},\ \|-\tilde{\mu}_1^{ik}z_0-\tilde{\mu}_2^{ik}z+C^T \tilde{\nu}^{ik}\|\le\lambda_{k}, \\
  \mu^{ik} \in \mathbb{R}_{+}^{3},
  \nu^{ik}\in \mathbb{R}^{l}_+,
  \tilde{\mu}^{ik} \in \mathbb{R}_{+}^{3},
  \tilde{\nu}^{ik}\in \mathbb{R}^{l}_+,
  V^{ik} \in \mathbb{R},
\end{array}
\right\}\\
&\qquad \qquad\qquad \qquad\qquad \qquad\qquad \qquad\qquad \qquad \qquad \qquad i=1,\cdots,N,~k=1,\cdots,\mathcal{K}, \\
&\ \ z\in Z,\lambda \in \mathbb{R}_{+}^{\mathcal{K}}.
\end{split}
\end{equation*}
\end{footnotesize}%
is an upper bound to that of problem ($P_{SSD}$).
\end{theorem}
\begin{proof}
Problem ($P_{SSD-U}$) is a reformulation of problem \eqref{222}, and thus an upper bound approximation of problem ($P_{SSD}$). The infimum in ($D_{SSD-1}^{ik}$) and ($D_{SSD-2}^{ik}$) can be reached by the corresponding minimization problems due to the closeness of the feasible sets, given that the finite optimal values exists.
\end{proof}

\subsection{\texorpdfstring{Sequential convex approximation for ($P_{SSD-U}$)}{ }}
We observe that bilinear terms $\mu_2^{ik} z$ and $\tilde{\mu}_2^{ik}z$ in problem ($P_{SSD-U}$) make it difficult to solve problem ($P_{SSD-U}$) directly. We apply a sequential convex approximation method to solve problem ($P_{SSD-U}$), see Algorithm \ref{alg:1}. The idea is to separate coupling variables. At each iteration,
we fix $z$ and optimize w.r.t. $\mu,\tilde{\mu}$; then fix  $\mu,\tilde{\mu}$ and optimize w.r.t. $z$. The sequential convex approximation method finally generates a sequence of decisions whose objective values converge to an upper bound of the optimal value of  ($P_{SSD-U}$).

\begin{algorithm}
\setstretch{1}
\caption{Sequential convex approximation}
\label{alg:1}
\begin{algorithmic}
\STATE \textbf{Start from} $z^{\iota} \in Z$,~${\iota}=1$.
\WHILE{${\iota} \ge 1$}
\STATE Solve problem ($P_{SSD-U}$) with
an additional constraint $z = z^{\iota}$.
Denote the optimal $\mu,\tilde{\mu}$ by $\mu^{{\iota}},\tilde{\mu}^{{\iota}}$, respectively.
\STATE Solve problem ($P_{SSD-U}$) with
additional constraints $\mu=\mu^{{\iota}},\tilde{\mu} = \tilde{\mu}^{{\iota}}$.
Denote the optimal $z$ by $z^{{\iota}+1}$.
\IF{$z^{{\iota}+1} = z^{{\iota}}$}
\STATE \text{Break}.\\
\ELSE{}
\STATE ${\iota} \leftarrow {\iota}+1$.
\ENDIF
\ENDWHILE
\end{algorithmic}
\end{algorithm}

\begin{proposition}\label{the:5}
  Suppose that the optimal value of problem ($P_{SSD}$) is finite. Given a starting point $z^1$. Algorithm \ref{alg:1} generates a sequence of decisions whose objective values converge to an upper bound of the optimal value of problem ($P_{SSD-U}$).
\end{proposition}

\begin{proof}
  Denote the feasible set of problem ($P_{SSD-U}$) by $\mathcal{F}_{U}$. We write all the decision variables excluding $z,\mu,\tilde{\mu}$ by $y$. We can thus write problem ($P_{SSD-U}$) in a compact form $\min \{f(z)|~(z,\mu,\tilde{\mu},y) \in \mathcal{F}_{U}\}$.

  Firstly, observe that each problem we solve in Algorithm \ref{alg:1} has an additional constraint compared with problem ($P_{SSD-U}$). Therefore, $f(z^{\iota}),~{\iota}=1,\cdots,$ are upper bounds to the optimal value of problem ($P_{SSD-U}$).

  Next, the sequence $\{f(z^{\iota})\}$ has a finite lower bound, the optimal value of problem ($P_{SSD}$).
  Thus in order to show the convergence of $\{f(z^{\iota})\}$, it is sufficient to prove that $\{f(z^{\iota})\}$ is nonincreasing.
  From Algorithm \ref{alg:1}, there exists $y'$ such that
  $(\mu^{\iota},\tilde{\mu}^{\iota},  y')=\argmin_{\mu,\tilde{\mu},  y} \{f(z^{\iota})|~(z^{\iota},\mu,\tilde{\mu}, y) \in \mathcal{F}_U\}$.
  It follows immediately that $(z^{\iota}, \mu,\tilde{\mu}, y') \in \mathcal{F}_U$.
  Also there exists $y''$ such that
   $(z^{{\iota}+1},y'') = \argmin_{z,y}\{ f(z)|~(z, \mu^{\iota},\tilde{\mu}^{\iota},y) \in \mathcal{F}_U \}$.
  Since   $(z^{\iota}, \mu^{\iota},\tilde{\mu}^{\iota}, y') \in \mathcal{F}_U$,
  we have
  $
  f(z^{{\iota}+1}) \le f(z^{{\iota}}).
  $
\end{proof}

Here it is necessary to point out that any element in the sequence of optimal values generated by Algorithm \ref{alg:1} is an upper bound of the optimal value of problem ($P_{SSD-U}$). Each problem we solve in Algorithm \ref{alg:1} is a second-order cone programming and thus is computationally tractable.

To conclude this section, we divide $\mathcal{R}$ into sub-intervals and exchange the order of the expectation operator and the supremum over each sub-interval to derive an upper bound approximation ($P_{SSD-U}$) for the distributionally robust {SSD} constrained optimization problem ($P_{SSD}$). We prove the convergence of the optimal value of the upper bound approximation problem and quantitatively estimate the approximation error. To cope with bilinear terms in problem ($P_{SSD-U}$), we apply the sequential convex approximation method, Algorithm \ref{alg:1}, to obtain an upper bound of the optimal value of problem ($P_{SSD-U}$).

\section{Numerical experiments}\label{4}

In this section, we present the results of numerical experiments to illustrate the validity and practicality of our lower and upper bound approximation methods for model ($P_{SSD}$).
The numerical experiments are carried out by calling the Gurobi solver in CVX package in MATLAB R2016a on a Dell G7 laptop with Windows 10 operating system, Intel Core i7 8750H CPU 2.21 GHz and 16 GB RAM.

\subsection{Case study: an illustrative numerical example}
We begin with a simple numerical example and examine the validation of the proposed lower and upper bound approximations.
Consider the following problem:
\begin{flalign}
  \min \quad  & \frac{1}{2}\|z\|_2 \nonumber\\
  \text{s.t.} \quad & \mathbb{E}_{P}[(\eta-z^T\xi)_{+}] \le \mathbb{E}_{P}[(\eta-z_0^T \xi)_{+}],~\forall \eta \in \mathbb{R},~\forall P \in \mathcal{Q},\label{equ:newexam}\\
  & z \in \mathbb{R}_2^+,\|z\|_1 \le 1. \nonumber
\end{flalign}
where $z_0=(1,0)^T$ and $\mathcal{Q}=\{P \in \mathscr{P}(\Xi):d_K(P,\widehat{P}_N)\le \epsilon\}$ is defined as that in \eqref{equ:Q}.
Here $\widehat{P}_N=\frac{1}{N}\sum_{i=1}^{N}\delta_{\widehat{\xi}_i}$ is the empirical distribution.
The support set is supposed to be $\Xi=\{(\xi_1,\xi_2)^T|~\xi_1 \in [0,250],~\xi_2 \in [0,500]\}$.
We set $\epsilon=10^{-5},N=10$ and the observed sample set $\{\widehat{\xi}_i\}_{i=1}^{10}$ consists of
$(0,0)^T$, $(250,0)^{T}$, $(0,500)^{T}$, $(100,100)^T$, $(200,200)^T$, $(100,0)^T$, $(200,0)^T$, $(0,100)^T$, $(0,200)^T$, $(200,500)^T$.

\begin{table}[h]
  \centering
  \caption{The optimal values and the optimal solutions of the lower and upper bound approximations to problem \eqref{equ:newexam}.}\label{tab:55}
  \scalebox{0.68}{
  \begin{tabular}{llllllllr}
    \toprule
    \multicolumn{4}{l}{lower bound approximation (($P_{SSD-L}$) or Algorithm \ref{alg:11})} & &  \multicolumn{3}{l}{upper bound approximation (Algorithm \ref{alg:1})} &  {Gap$\qquad$} \\
    \cline{1-4}
    \cline{6-8}
     {$\mathcal{N}$} & {$\mathcal{M}$} & Optimal value & Optimal solution & & $\mathcal{K}$ & Optimal value & Optimal solution \\
    \hline
    100 & 100 & 0.2922 & $(0.4229,0.4027)^T$ & & 10 & 0.4097 & $(0.8010,0.1564)^T$  & {40.2122\%}\\
    200 & 200 & 0.2964 & $(0.4266,0.4077)^T$ & & 11 & 0.3044 & $(0.4590,0.4002)^T$  & {2.6991\%}\\
    300 & 300 & 0.3014 & $(0.4423,0.4014)^T$ & & 12 &0.3025  & $(0.4653,0.3868)^T$  & {0.3650\%}\\
    \bottomrule
  \end{tabular}}
\end{table}
We get the lower bound approximation by solving the linear programming formulation ($P_{SSD-L}$) or Algorithm \ref{alg:11} and obtain the upper bound approximation by Algorithm \ref{alg:1}. The optimal values and the optimal solutions are shown in Table \ref{tab:55}. We also calculate the relative gaps of the optimal values of the lower and upper bound approximations (i.e., Gap$=|\frac{\text{upper}-\text{lower}}{\text{lower}}|$). From Table \ref{tab:55}, we can see that the relative gap between the optimal values of the lower and upper bound approximations decreases quickly to $0$ with the increase of sample sizes $\mathcal{N},\mathcal{M}$ and the interval number $\mathcal{K}$, which verifies the validation of the proposed approximation methods.

\subsection{Case study: a practical portfolio selection problem}
We consider a financial application of model ($P_{SSD}$) to the portfolio selection problem with distributionally robust SSD constraints:
\begin{equation}\label{problem5-2}\begin{array}{cl}
\min \limits_{z \in Z}  \quad &  \mathbb{E}_{\widehat{P}_N}[-z^T \xi] \\
\quad\text{s.t.} \,\quad & \mathbb{E}_{P}[(\eta-z^T\xi)_{+}] \le \mathbb{E}_{P}[(\eta-z_0^T \xi)_{+}],~\forall \eta \in \mathbb{R},~\forall P \in \mathcal{Q},
\end{array}\end{equation} 
where $Z=\{z\in \mathbb{R}^n|~z\ge0,\sum_{i=1}^{n}z_i=1\}$. Problem \eqref{problem5-2} is inspired by \cite[Example 4.2]{Xu}. The difference between problem \eqref{problem5-2} and that in \cite[Example 4.2]{Xu} lies in the construction method of the ambiguity set. In \cite[Example 4.2]{Xu}, the ambiguity set $\mathcal{Q}$ is determined by first two order moment information, while in problem \eqref{problem5-2}, $\mathcal{Q}$ is a Wasserstein ball.

We use the same historical annual return rate data of eight risky assets as that in \cite[Table 8.1]{D1} (with a total of $N=22$ years).
We choose the equally weighted portfolio as the benchmark portfolio $z_0$.
We select $z_0$ as the starting point $z^1$ in Algorithm \ref{alg:1}.

In what follows, we show the numerical results emphatically illustrating from the following aspects:
the convergence of the lower and upper bound approximations with respect to the sample sizes and the interval number,
the price of introducing distributional robustness in SSD constraints,
and the influence of the robust radius.

\subsubsection{Convergence of lower and upper bounds}
Firstly, we 
demonstrate the convergence of the lower bound approximation with respect to the sample sizes $\mathcal{N},~\mathcal{M}$ and the decreasing trend of the upper bound approximation when the interval number $\mathcal{K}$ increases.
We fix the robust radius $\epsilon=10^{-4}$.

For the lower bound approximation, we consider the cases with the sample sizes being $\mathcal{N}=\mathcal{M}=40,60,80,100,120$.
To make fair comparison later in Section \ref{subsec:4.1} with the portfolio optimization problem with classic (non-robust) SSD constraints, we let $\Xi_{\mathcal{N}}$ contain all the historical annual return rates $\{\widehat{\xi}_i\}_{i=1}^{N}$ from \cite[Table 8.1]{D1} and $\varGamma_{\mathcal{M}}$ contain $\{z_0^T\widehat{\xi}_i\}_{i=1}^{N}$.
For the upper bound approximation, we consider the cases with the interval number being $\mathcal{K}=1,2,4,8,12$, respectively. Figure \ref{fig:conv} shows the convergence trend of the optimal values of the lower and upper bound approximations for problem \eqref{problem5-2}.

\begin{figure}[h]
	\centering
\scalebox{0.55}{
	\includegraphics{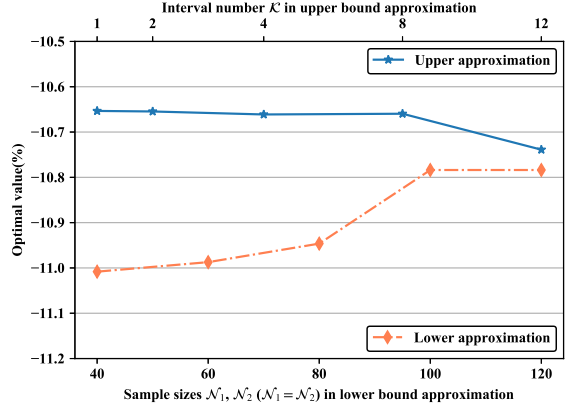}}
	\caption{Optimal values of the lower bound approximation with respect to $\mathcal{N}, \mathcal{M}$ and that of the upper bound approximation with respect to $\mathcal{K}$.}\label{fig:conv}
\end{figure}

From Figure \ref{fig:conv}, we can observe that the lower bound monotonuously increases with the increase of the sample sizes $\mathcal{N},~\mathcal{M}$ and the upper bound decreases with the increase of the interval number $\mathcal{K}$. The gap between the lower and uppers approaches $0$.
These observations verify the quantitative convergency established in Theorem \ref{the:0927-1} and Theorem \ref{pro:cite}.
To see more details, we present in Table \ref{tab:n1} the optimal values and the optimal solutions obtained from the lower and upper bound approximations.

\begin{table}[h]
  \centering
  \caption{Optimal values and the optimal solutions of the lower bound approximation with respect to $\mathcal{N}, \mathcal{M}$,
  and those of the upper bound approximation with respect to $\mathcal{K}$.}
  \label{tab:n1}
  \scalebox{0.62}{
  \begin{tabular}{llllllll}
    \toprule
    \multicolumn{3}{l}{lower bound approximation (($P_{SSD-L}$) or Algorithm \ref{alg:11})} & $~$ & \multicolumn{2}{l}{upper bound approximation (Algorithm \ref{alg:1})} & &  \\
    \cline{1-3}
    \cline{5-6}
    $\mathcal{N}$ & $\mathcal{M}$ & Optimal value(\%)& & $\mathcal{K}$ & Optimal value(\%) & &Gap \\
    & & Optimal solution & & & Optimal solution \\
    \midrule
    40  &  40  &  -11.0082 &  & 1 & -10.6534  & & 3.2231\% \\
    & & (0.000,0.000,0.068,0.188,0.000,0.391,0.231,0.122) & & & (0.125,0.125,0.125,0.125,0.125,0.125,0.125,0.125) \\
    60  &  60  &  -10.9872 &  & 2 & -10.6543  & & 3.0299\%\\
    & & (0.000,0.038,0.000,0.269,0.000,0.354,0.213,0.126) & & & (0.125,0.124,0.124,0.127,0.125,0.126,0.125,0.125) \\
    80  &  80  &  -10.9463 &  & 4 & -10.6546  & & 2.6648\%\\
    & & (0.000,0.006,0.094,0.138,0.036,0.389,0.215,0.123) & & & (0.124,0.124,0.123,0.127,0.124,0.127,0.125,0.125) \\
    100 &  100 &  -10.7838 &  & 8 & -10.6551  & & 1.1935\%\\
    & & (0.000,0.018,0.168,0.000,0.131,0.384,0.172,0.126) & & & (0.124,0.124,0.123,0.128,0.124,0.127,0.125,0.125) \\
    120 &  120 &  -10.7838 &  & 12& -10.7389  & & 0.4164\%\\
    & & (0.000,0.018,0.168,0.000,0.131,0.384,0.172,0.126) & & & (0.075,0.067,0.005,0.274,0.087,0.238,0.125,0.129) \\
    \bottomrule
  \end{tabular}}
\end{table}

From Table \ref{tab:n1}, we can see the changing trend of the optimal portfolios of the lower and upper bound approximations. Especially for the upper bound approximation, the optimal portfolio under $\mathcal{K}=1$ is the equally weighted portfolio, while the optimal portfolio under $\mathcal{K}=12$ is quite different from the equally weighted portfolio and approaches the optimal portfolios obtained from the lower bound approximation.
We observe from Table \ref{tab:n1} that the lower and upper bounds we finally obtain are not equal.
This happens because when implementing the upper bound approximation we can only have finite sub-intervals and thus a gap is induced whenever we exchange the order of operators $\sup_{\eta \in [\uline{\eta}_{k},\bar{\eta}_{k}]}$ and $\mathbb{E}_P$.
We calculate the relative gap between the upper bound with $\mathcal{K}=12$ and the lower bound with $\mathcal{N}=120,\mathcal{M}=120$, which is only $|\frac{10.7838-10.7389}{-10.7838}|=0.4164\%$. This is quite satisfactory for real applications.

\subsubsection{Price of distributional robustness}\label{subsec:4.1}

To examine the price of introducing distributional robustness, we compare the numerical result of problem \eqref{problem5-2} with that of classic SSD constrained portfolio optimization problem
\begin{equation}\label{equ:comparison model}
  \min  \left\{ \mathbb{E}_{\widehat{P}_N}[-z^T \xi] | ~z \in Z, z^T \xi \succeq_{(2)}^{\widehat{P}_N} z_0^T \xi \right\}.
\end{equation}
Table \ref{tab:3} reports the comparative results of the optimal expected return rates, which are absolute values of the optimal values of problems \eqref{problem5-2} and \eqref{equ:comparison model} since they are minimization problems.

\begin{table}[ht]
\caption {Optimal expected return rates (absolute value of the optimal value) of the lower and upper bound approximations to problem \eqref{problem5-2}, optimal expected return rate of problem \eqref{equ:comparison model}, and expected return rate of $z_0$.}\label{tab:3}
\centering
\scalebox{0.75}{
\begin{tabular}{lllll}
\toprule
\multicolumn{2}{l}{Portfolio optimization problem} & &{Expected return rate(\%)}     &                   \\
\cline{1-2} \cline{4-5}
\multirow{2}{*}{ \eqref{problem5-2}}
     &  lower bound approximation ($\mathcal{N}=\mathcal{M}=120$)   &  &   10.7838                         \\
     &  upper bound approximation ($\mathcal{K}=12$)  &  &   10.7389                \\
\cline{1-2} \cline{4-5}
\multicolumn{2}{l}{{ \eqref{equ:comparison model}}}
                                    &  &   11.0082                    \\
\multicolumn{2}{l}{{Benchmark}}     &  &   10.6534         \\
\bottomrule
\end{tabular}}
\end{table}

From Table \ref{tab:3}, we can see that both the lower and upper bound approximations to problem \eqref{problem5-2} with distributionally robust SSD constraints derive a smaller optimal expected return rate than problem \eqref{equ:comparison model} with classic SSD constraints. Therefore, the optimal expected return rate of problem \eqref{problem5-2} must be smaller than that of problem \eqref{equ:comparison model}. As we expected, considering the distributionally robust ambiguity in SSD constraints induces a more conservative solution. It can also be seen from Table \ref{tab:3} that the expected return rates of the lower and upper bound approximations are larger than that of the benchmark portfolio, which means that model \eqref{problem5-2} derives a portfolio better than the benchmark portfolio in sense of the expected return rate. These numerical results demonstrate that introducing distributional robustness brings in conservation without loss of stochastic dominance.

\subsubsection{Influence of the robust radius}
Finally, we briefly examine the impact of robust radius on the
lower and upper bound approximations to the
portfolio optimization problem \eqref{problem5-2}. 
The optimal values of the lower and upper bound approximations
under different robust radii
are shown in Table \ref{tab:1}.
\begin{table}[h]
\caption {Optimal values of the lower and upper bound approximations,
and their relative gaps
with respect to different robust radii.}\label{tab:1}
\centering
\scalebox{0.75}{
\begin{tabular}{llllr}
\toprule
{Robust radius}  & \multicolumn{2}{l}{Optimal values (\%)}& &{Gap$\qquad$}\\
\cline{2-3}
{$\epsilon$}          &  lower bound approximation                      &  upper bound approximation   \\
\midrule
$10^{-5}$   &       -10.8775                            &       -10.8268  && 0.4661\%\\
$10^{-4}$   &       -10.7838                            &       -10.7389  && 0.4164\%\\
$10^{-3}$   &       -10.7836                            &       -10.6536  && 1.2055\%\\
$10^{-2}$   &       -10.7823                            &       -10.6535  && 1.1946\%\\
0.1         &       -10.7689                            &       -10.6534  && 1.0725\%\\
0.5         &       -10.6885                            &       -10.6534  && 0.3284\%\\
1           &       -10.6534                            &       -10.6534  && 0\%\\
\bottomrule
\end{tabular}}
\end{table}

We can see from Table \ref{tab:1} that, as is expected both the optimal values of the lower and upper bound approximations of problem \eqref{problem5-2} are monotonously increasing, which implies that the optimal value of problem \eqref{problem5-2} increases as the robust radius increases.
Table \ref{tab:1} also tells us that choosing a proper robust radius is a crucial issue in distributionally robust SSD constrained problems. For robust radius $\epsilon \ge 0.1$, the upper bound coincides with $\mathbb{E}_{\widehat{P}_N}[-z_0^T\xi]$, while for $\epsilon \le 10^{-2}$ both the lower and upper bound approximations derive optimal portfolios better than the benchmark portfolio. 


\section{Conclusion}\label{5}

We consider a distributionally robust SSD constrained optimization problem, where the true distribution of the uncertain parameters is ambiguous. The ambiguity set contains those probability distributions close to the empirical distribution under the Wasserstein distance.

We propose two approximation methods to obtain bounds on the optimal value of the original problem. We adopt the sample approximation approach to develop a linear programming formulation to obtain a lower bound approximation for the problem. The lower bound approximation can be easily solved by using linear programming formulation or by the cutting-plane method.
Moreover, {we establish the quantitative convergency for the lower bound approximation problem.} We also develop an upper bound approximation and {quantitatively estimate the approximation error} between the optimal value of the upper bound approximation and that of the original problem. We propose a novel split-and-dual decomposition framework to reformulate distributionally robust SSD constraints. The upper bound approximation problem can be solved by a sequence of second-order cone programming problems. We carry out numerical experiments on a portfolio optimization problem to illustrate our lower and upper bound approximation methods.

One of future research topics would be modifying the design of cutting-planes to solve the lower bound approximation problem more efficiently. {While for the upper bound approximation, it is interesting to investigate the critical number of intervals for further enhancing the practicality of the approximation scheme.} Besides, finding efficient approximation and solution methods for distributionally robust multivariate robust SSD constrained optimization is also a promising topic.

\section*{Funding}
This research was supported by the National Natural Science Foundation of China under grant numbers 11991023, 11991020,  11735011, and 11901449.

\begin{spacing}{1}
\bibliographystyle{plain}
\bibliography{reference}
\end{spacing}
\end{document}